\documentclass[final,3p,times]{elsarticle}

\usepackage{amsmath,amssymb,amscd,mathtools}
\usepackage{enumitem}
\usepackage{graphicx}
\usepackage{xcolor}
\usepackage{subcaption}
\usepackage{algorithm}
\usepackage{algpseudocode}
\usepackage{diagbox}
\usepackage[colorlinks=true, allcolors=blue]{hyperref}
\usepackage{xparse}

\usepackage{tikz}
\usepackage{pgfplots}
\usepackage{pgfplotstable}

\usepackage{booktabs}
\usepackage{siunitx}

\biboptions{sort&compress}
\usetikzlibrary{matrix,backgrounds}
\pgfdeclarelayer{myback}
\pgfsetlayers{myback,background,main}

\tikzset{mycolor/.style = {line width=1bp,color=#1}}%
\tikzset{myfillcolor/.style = {draw,fill=#1}}%

\NewDocumentCommand{\highlight}{O{blue!40} m m}{%
\draw[mycolor=#1] (#2.north west)rectangle (#3.south east);
}

\NewDocumentCommand{\fhighlight}{O{blue!40} m m}{%
\draw[myfillcolor=#1] (#2.north west)rectangle (#3.south east);
}

\newcommand{\A}{{\mathcal A}}

\def\bff{{\mathbf f}}
\def\bg{{\mathbf g}}

\def\bsigma{{\boldsymbol \sigma}}

\def\diag{{\bf diag }}

\def\interpolate{{\blkMat{P}_{\ell+1}^\ell}}
\def\restrict{{\blkMat{R}_\ell^{\ell+1}}}

\def\inject{{\blkMat{Q}_\ell^{\ell+1}}}

\newcommand{\tensorOne}[1]{\boldsymbol{#1}}
\newcommand{\tensorTwo}[1]{\mathbb{#1}}

\renewcommand{\Vec}[1]{%
  \ifcat\noexpand#1\relax 
    \boldsymbol{#1}
  \else
    \mathbf{#1}
  \fi
}
\newcommand{\Mat}[1]{#1}

\newcommand{\blkVec}[1]{\Vec{#1}}
\newcommand{\blkMat}[1]{\Vec{#1}}

\newdefinition{rmk}{Remark}
\newdefinition{definition}{Definition}

\usepackage{amsthm}
        \theoremstyle{plain}
        \newtheorem{thm}{Theorem}
        \newtheorem{proposition}[thm]{Proposition}
       	\newtheorem{lemma}{Lemma}

\definecolor{butter1}{rgb}{0.988,0.914,0.310}
\definecolor{chocolate1}{rgb}{0.914,0.725,0.431}
\definecolor{chameleon1}{rgb}{0.541,0.886,0.204}
\definecolor{skyblue1}{rgb}{0.247,0.524,0.912}
\definecolor{applegreen}{rgb}{0.55, 0.71, 0.0}
\definecolor{blue-green}{rgb}{0.0, 0.87, 0.87}
\definecolor{plum1}{rgb}{0.678,0.498,0.659}
\definecolor{scarletred1}{rgb}{0.937,0.161,0.161}

\journal{arXiv}

\begin{document}
\title{Nonlinear multigrid based on local spectral coarsening for heterogeneous diffusion problems}

\begin{frontmatter}
\author[casc]{Chak Shing Lee\corref{cor1}}
\ead{cslee@llnl.gov}
\author[total,aeed,stanford]{Fran\c cois Hamon}
\ead{francois.hamon@total.com}
\author[aeed]{Nicola Castelletto}
\ead{castelletto1@llnl.gov}
\author[casc,psu]{Panayot S. Vassilevski}
\ead{vassilevski1@llnl.gov, panayot@pdx.edu}
\author[aeed]{Joshua White}
\ead{white230@llnl.gov}

\cortext[cor1]{Corresponding author.}
\address[casc]{Center for Applied Scientific Computing, Lawrence Livermore National Laboratory, Livermore, CA 94550, USA}
\address[total]{Total E\&P Research and Technology, Houston, TX 77002, USA}
\address[aeed]{Atmospheric, Earth, and Energy Division, Lawrence Livermore National Laboratory, Livermore, CA 94550, USA}
\address[stanford]{Department of Energy Resources Engineering, Stanford University, Stanford, CA 94305, USA}
\address[psu]{Fariborz Maseeh Department of Mathematics and Statistics, Portland State University, Portland, OR 97201, USA} 
\begin{abstract}
This work develops a nonlinear multigrid method for diffusion problems discretized by cell-centered finite
volume methods on general unstructured grids. 
The multigrid hierarchy is constructed algebraically using aggregation of degrees of freedom and spectral decomposition of reference linear operators associated with the aggregates. 
For rapid convergence, it is important that the resulting coarse spaces have good approximation properties. In our approach, the approximation quality can be directly improved by including more spectral degrees of freedom in the coarsening process.
Further, by exploiting local coarsening and a piecewise-constant approximation when evaluating the nonlinear component, the coarse level problems are assembled and solved without ever re-visiting the fine level, an essential element for multigrid algorithms to achieve optimal scalability. 
Numerical examples comparing relative performance of the proposed nonlinear multigrid solvers with standard single-level approaches---Picard's and Newton's methods---are presented.
Results show that the proposed solver consistently outperforms the single-level methods, both in efficiency and robustness.
\end{abstract}


\end{frontmatter}


\allowdisplaybreaks

\section{Introduction}\label{sec:intro}
This paper explores multilevel nonlinear solution strategies for diffusion processes, with an emphasis on fluid flow through highly heterogeneous porous media.  
The goal is to develop a scalable strategy, in the sense that the computational cost is proportional to the problem size.
This requirement is essential for enabling extreme-scale simulations in scientific computing.

For nonlinear diffusion problems, scalability is most frequently achieved by exploiting multilevel solvers.
In practice, two approaches are commonly adopted.
The first strategy is to apply a global linearization method---e.g. Picard's or Newton's method---and solve the resulting linearized systems using multilevel linear solvers.
This approach is often the method of choice for problems encountered in science and engineering applications; see, for example, \cite{keyes06}.  
The second strategy, and the one of interest here, is to apply the multigrid idea directly to the nonlinear problem.
Recent works \cite{luo2015, christensen16, toft18} have demonstrated that well-designed nonlinear multigrid can be more efficient and robust in certain applications.

A key requirement for nonlinear multigrid is the ability to construct coarse spaces with good approximation properties.
This is usually not an issue for geometric multigrid, which is defined on a sequence of successively refined grids.
However, for PDEs discretized on general unstructured meshes, the coarse grids are typically obtained by agglomerating fine-grid elements.  This can result in coarse elements with very complex shapes.
In the context of finite-element and mixed finite-element discretizations \cite{dumett02,jones03,christensen18}, operator-dependent interpolation operators have been obtained in previous works by restricting polynomials on coarse element boundaries and then extending them into neighboring coarse elements.
In this paper, we start from a finite-volume discretization, in which case higher-order polynomials are not in the natural approximation space.
Instead of polynomials, we therefore employ a local spectral coarsening method recently developed for graph-Laplacians \cite{barker17,ml-spectral-coarsening}. Coarse degrees of freedom are defined using the eigenvectors of the local operator.

The current work proposes a method for nonlinear diffusion problems with several compelling features:
\begin{enumerate}
\item The algorithm can be applied to nonlinear diffusion problems on general unstructured grids.
\item The coarsening is done in a mixed setting (both pressure and flux are coarsened).
\item The approximation properties of coarse spaces can be enhanced by flexible local enrichment.
\item Coarse level problems may be assembled and solved without re-visiting finer levels during the multigrid cycle.
\end{enumerate}
The final point is a crucial property for achieving scalability, as fine-level computations can form a significant bottleneck.
The coarse assembly and solution is achieved using the aggregation-based local coarsening and a piecewise-constant approximation of the solution when evaluating the nonlinear component.
Moreover, since the coarsening is done in a local fashion, the global coarse nonlinear operator can be applied by utilizing the local coarse operator and the local-to-global map.
For the diffusion problem considered in this paper, this allows us to derive an explicit formula for the Jacobian matrix in the Newton iterations, as well as an algebraic hybridization solver for the linearized problems.

Numerical examples demonstrate that the nonlinear multigrid solver presented in this paper is more robust than standard single-level solvers, in the sense that it converges in fewer iterations for strongly nonlinear problems.
We also show that the proposed algorithm is highly scalable, significantly reducing wallclock time compared to its single-level counterparts.
%

The remainder of the paper is organized as follows. Section~\ref{sec:model} presents the nonlinear PDE and its finite-volume discretization.
The components of the proposed nonlinear multigrid algorithm are discussed in detail in Section~\ref{sec:multigrid}.
We then present numerical examples comparing the proposed multigrid solver to single-level Picard's and Newton's methods in Section~\ref{sec:numerics}, followed by concluding remarks in Section~\ref{sec:conclude}.

\section{Model problem}\label{sec:model}
We consider a nonlinear diffusion problem in a heterogeneous domain, such as those typically encountered when modeling fluid flow in porous media.
The central PDE is a mass conservation equation for incompressible single-phase flow \cite{coats2000}, in which the diffusion tensor $\tensorTwo{K}$ depends on the pressure $p$ in a differentiable manner.
The strong form of the boundary value problem (BVP) is stated as follows:
\begin{subequations}
\begin{align}
  \intertext{Given $f: \Omega \rightarrow \mathbb{R}$, $g_D: \partial \Omega_D \rightarrow \mathbb{R}$ and $g_N: \partial \Omega_N \rightarrow \mathbb{R}$, find $p: \overline{\Omega} \rightarrow \mathbb{R}$ such that}
	\nabla \cdot \tensorOne{q} (\tensorOne{x}, p )
	&=
	f(\tensorOne{x}) \,,
	&& \tensorOne{x} \in \Omega
	&& \mbox{(mass conservation)},
	\label{eq:pres_equation} \\
	\tensorOne{q} (\tensorOne{x}, p )
	&=
  -\mathbb{K}(\tensorOne{x}, p ) \cdot \nabla p (\tensorOne{x}) \,,
	&& \tensorOne{x} \in \Omega
	&& \mbox{(Darcy's law)} ,
	\label{eq:Darcy_law} \\
  p(\tensorOne{x})
  &=
  g_D(\tensorOne{x}) \,,
  && \tensorOne{x} \in \partial \Omega_D
  && \mbox{(prescribed boundary pressure)},
  \label{eq:BC_dir}\\
  - \tensorOne{q}(\tensorOne{x}) \cdot \tensorOne{n}(\tensorOne{x})
  &=
  g_N(\tensorOne{x}) \,,
  && \tensorOne{x} \in \partial \Omega_N
  && \mbox{(prescribed boundary flux)}.
  \label{eq:BC_neu}
\end{align}
\label{eq:model}\null
\end{subequations}
Here, $\Omega$ is a bounded and connected polygonal domain in $\mathbb{R}^d$, $d = 2, 3$.
The domain boundary $\partial \Omega$ is split into two disjoint portions such that $\partial \Omega = \overline{\partial \Omega_D \cup \partial \Omega_N}$,  with $\tensorOne{n}$ the unit outward normal to $\partial\Omega$ and $\tensorOne{x}$ the position vector in $\mathbb{R}^d$.
We assume that $\tensorTwo{K}$---representing the medium's absolute permeability tensor divided by fluid viscosity---can be expressed in the functional form

\begin{equation}
  \mathbb{K}(\tensorOne{x}, p ) = \mathbb{K}_0 (\tensorOne{x}) \kappa(p),
  \label{eq:kappa}
\end{equation}

\noindent
where $\mathbb{K}_0$ is a heterogeneous and anisotropic diffusion tensor that is independent of $p$, while $\kappa(p)$ is a dimensionless pressure-dependent scalar multiplier.
To demonstrate the robustness of the proposed nonlinear solution algorithm, we will consider several highly nonlinear permeability-pressure functional relationships.
  As a representative example, in the first two numerical benchmarks below we define the pressure-dependent scalar multiplier as
  \begin{equation}
    \kappa(p) := e^{\alpha p},
    \label{eq:permeability_pressure_relationship}
  \end{equation}
  where $\alpha > 0$ is a user-defined parameter that controls the strength of the nonlinearity.

\subsection{Finite Volume discretization}\label{sec:FV_tpfa}
The BVP \eqref{eq:model} is discretized by a cell-centered finite-volume (FV) method on a conforming triangulation of the domain.
First, some notation is defined.
Let $\mathcal{T}$ be the set of cells in the computational mesh such that $\overline{\Omega} = \sum_{\tau \in \mathcal{T}} \overline{\tau}$.
For a cell $\tau_K \in \mathcal{T}$, with $K$ a global index, let $|\tau_K|$ denote the $d-$measure, $\partial \tau_{K} = \overline{\tau}_{K} \setminus \tau_{K}$ the boundary, $\tensorOne{x}_{K}$ the barycenter, and $\tensorOne{n}_{K}$ the outer unit normal vector.
Let $\mathcal{E}$ be the set of interfaces---i.e., edges in $\mathbb{R}^2$ or faces in $\mathbb{R}^3$---in the computational mesh such that $\mathcal{E} =\mathcal{E}_{\text{int}} \cup \mathcal{E}_D \cup \mathcal{E}_N$, with $\mathcal{E}_{\text{int}}$ (respectively $\mathcal{E}_{D}$, $\mathcal{E}_{N}$) the set of interfaces included in $\Omega$ (respectively $\partial \Omega_D$, $\partial \Omega_N$).
An internal interface $\varepsilon$ shared by cells $\tau_K$ and $\tau_L$ is denoted as $\varepsilon_{K,L} = \partial \tau_K \cap \partial \tau_L$, with the indices $K$ and $L$ such that $K < L$.
A boundary interface belongs to a single cell $\tau_K$ and is denoted $\varepsilon_K$.
The $(d-1)$-measure of an interface is $|\varepsilon|$.
A unit vector $\tensorOne{n}_{\varepsilon}$ is introduced to define a unique orientation for every interface.
We set $\tensorOne{n}_{\varepsilon} = \tensorOne{n}_{K}$ both for internal and boundary interfaces.
To indicate the mean value of a quantity $(\cdot)$ over an interface $\varepsilon$ or a cell $\tau_K$, we use the notation $(\cdot)_{\left| \right. \varepsilon}$ and $(\cdot)_{\left| \right. K}$, respectively.

The derivation of the FV form   of \eqref{eq:model} consists of writing the pressure equation for a generic cell in integral form.
Making use of Gauss' divergence theorem, the following set of balance equations is obtained
\begin{equation}
\sum_{\varepsilon \in \partial \tau_K} \int_{\varepsilon} \tensorOne{q} \cdot \tensorOne{n}_{K} \;d\Gamma = \int_{\tau_K} f \;d\Omega \qquad \forall \tau_{K} \in \mathcal{T}. 
\label{eq:mass_balance}
\end{equation}
The FV scheme requires discrete approximations for the pressure field and the Darcy flux through cell interfaces.
We consider a piecewise constant approximation for pressure.
For each cell $\tau_{K} \in \mathcal{T}$, we introduce one degree of freedom, $p_K$.
We denote by $\sigma_{\varepsilon}$ a numerical flux approximating the Darcy flux through $\varepsilon \in \mathcal{E}$, i.e. $\sigma_{\varepsilon}  \approx \int_{\varepsilon} \tensorOne{q} \cdot \tensorOne{n}_{\varepsilon} \;d\Gamma$.
In general, $\sigma_{\varepsilon}$ can be computed using a suitable functional dependence on cell pressure values as well as Dirichlet and Neumann boundary data.
In this work we consider a two-point flux approximation (TPFA).
Introducing a collocation point $\tensorOne{x}_{\varepsilon}$ for every $\varepsilon \in \mathcal{E}$, which is used to enforce point-wise pressure continuity across internal interfaces, the TPFA numerical flux is defined as \cite{EymGalHer00}
\begin{equation}
  \sigma_{\varepsilon} =
  \begin{dcases}
  - \Upsilon_{KL} (p_{L} - p_{K}), &\text{if } \varepsilon = \varepsilon_{K,L} \in \mathcal{E}_{int}, \\
  - \Upsilon_{K,\varepsilon} ( g_{D \left| \right. \varepsilon}, - p_{K}), &\text{if } \varepsilon = \varepsilon_{K} \in \mathcal{E}_D, \\  
  - | \varepsilon| \, g_{N \left| \right. \varepsilon}, &\text{if } \varepsilon = \varepsilon_{K} \in \mathcal{E}_N,
  \end{dcases}
  \label{eq:TPFA_flux}
\end{equation}
where the transmissibility $\Upsilon_{KL} = \left( \frac{1}{\Upsilon_{K,\varepsilon}} + \frac{1}{\Upsilon_{L,\varepsilon}}  \right)^{-1}$ is a weighted harmonic average of the one-sided transmissibility $\Upsilon_{K,\varepsilon}$ and $\Upsilon_{L,\varepsilon}$.
The one-sided transmissibility coefficients---also known as half transmissibilities---read
\begin{align}
  \Upsilon_{i,\varepsilon} &=
  \kappa(p_{i}) \overline{\Upsilon}_{i,\varepsilon},
  &
  \overline{\Upsilon}_{i,\varepsilon} &=
  |\varepsilon| \frac{\tensorOne{n}_{i} \cdot \mathbb{K}_{0 \left| \right. i} \cdot (\tensorOne{x}_{\varepsilon} - \tensorOne{x}_{i})}{||\tensorOne{x}_{\varepsilon} - \tensorOne{x}_{i}||_2^2},
  &
  i &= \{K, L \},
  \label{eq:half-transmissibility}
\end{align}
where $\overline{\Upsilon}_{i,\varepsilon}$ indicates the geometric (constant) term of $\Upsilon_{i,\varepsilon}$.
Denoting by $\sigma_{K,\varepsilon} = \tensorOne{n}_{K}(\tensorOne{x}_{\varepsilon}) \cdot \tensorOne{n}_{\varepsilon}  \, \sigma_{\varepsilon}$ the outgoing flux across $\varepsilon$ for cell $\tau_{K}$, the FV form of \eqref{eq:mass_balance} may be stated as \cite{EymGalHer00}
\begin{equation}
  \sum_{\varepsilon \in \partial \tau_{K} } \sigma_{K,\varepsilon} = |\tau_{K}| f_{\left. \right| K}, \qquad \forall \tau_{K} \in \mathcal{T}.
  \label{eq:mass_balance_FV}
\end{equation}

We now introduce coefficient vectors $\Vec{p} = (p_{K})_{\tau_{K} \in \mathcal{T}}$ and $\Vec{\sigma} = (\sigma_{\varepsilon})_{\varepsilon \in \mathcal{E}}$---i.e. the collection of all pressure and flux degrees of freedom, respectively.
To enable the coarsening strategy used in the proposed multilevel nonlinear solver, we rewrite \eqref{eq:mass_balance_FV} as an augmented nonlinear system
\begin{equation}
  \begin{bmatrix}
    M(\Vec{p}) & D^T\\
    D      & 0
  \end{bmatrix} 
  \begin{bmatrix}
    \bsigma\\
    \blkVec{p}
  \end{bmatrix} 
  =
  -
  \begin{bmatrix}
    \bg\\
    \bff
  \end{bmatrix},
  \label{eq:FV_matrix}
\end{equation}
with the following matrix and vector definitions:
\begin{subequations}
\begin{align}
[M(\Vec{p})]_{ij} &=
\begin{cases}
  \left( \frac{1}{\kappa(p_{K}) \overline{\Upsilon}_{K,\varepsilon_i}} + \frac{1}{\kappa(p_{L}) \overline{\Upsilon}_{L,\varepsilon_i}} \right) \delta_{ij}, & \text{if } \varepsilon_i = \varepsilon_{K,L} \in \mathcal{E}_{\text{int}}, \\
  \frac{1}{\kappa(p_{K}) \overline{\Upsilon}_{K,\varepsilon_i}} \delta_{ij}, & \text{if } \varepsilon_i = \varepsilon_{K} \in \mathcal{E}_D, \\
  \delta_{ij}, & \text{if } \varepsilon_i = \varepsilon_{K} \in \mathcal{E}_N,
\end{cases} 
&
[D^T]_{ij} &=
\begin{cases}
  \delta_{iL} - \delta_{iK}, & \text{if } \varepsilon_i = \varepsilon_{K,L} \in \mathcal{E}_{\text{int}}, \\
  -\delta_{iK}, & \text{if }  \varepsilon_i = \varepsilon_{K} \in \mathcal{E}_{D}, \\
  0, & \text{if } \varepsilon_i = \varepsilon_{K} \in \mathcal{E}_N,
\end{cases} \\
[ \bg ]_i &=
\begin{cases}
  0, & \text{if } \varepsilon_i \in \mathcal{E}_{\text{int}}, \\
  g_{D,\varepsilon_i}, & \text{if } \varepsilon_i \in \mathcal{E}_D, \\
  |\varepsilon_i| g_{N,\varepsilon_i}, & \text{if } \varepsilon_i \in \mathcal{E}_{N},
\end{cases} 
&
[ \mathbf{f} ]_K &=
  |\tau_K| f_{\left. \right| \tau_K} + \sum_{\varepsilon \in \partial \tau_K \cap \mathcal{E}_N} |\varepsilon| g_{N,\varepsilon} ,
\end{align}
\end{subequations}
where $\delta_{ij}$ is the Kronecker delta.
Note that $M$ is diagonal, while $D$ and $D^T$ resemble the discrete divergence and gradient operators respectively. 

\begin{rmk}
Different strategies can be used to choose collocation points $\tensorOne{x}_{\varepsilon}$.
Following \cite{KarDur16}, for internal interfaces we select $\tensorOne{x}_{\varepsilon}$ as the intersection $\varepsilon_{K,L}$ and the line connecting $\tensorOne{x}_{K}$ and $\tensorOne{x}_{L}$.
For a boundary interface, $\tensorOne{x}_{\varepsilon}$ is chosen as the orthogonal projection of $\tensorOne{x}_{K}$ on $\varepsilon_K$.
Note that in $\mathbb{R}^3$ we associate a plane characterized by a mean normal $\tensorOne{n}_{\varepsilon}$ to any non-planar face.
\end{rmk}

\begin{rmk}
Instead of \eqref{eq:FV_matrix}, a standard single-level cell-centered FV discretization would typically assemble the reduced system $\Mat{A}( \Vec{p} ) = \Vec{f}$, with $\Mat{A}( \Vec{p} ) = D (M (\Vec{p}) )^{-1} D^T \Vec{p} - D (M (\Vec{p}) )^{-1} \Vec{g}$.
\end{rmk}

\begin{rmk}
Because of the two-point structure, a linear TPFA approach produces a monotone finite-volume scheme characterized by linear systems with an M-matrix \cite{BerPle94} and small stencils.  
Unfortunately, the lack of consistency in linear TPFA may produce inaccurate results in the presence of distorted grids or highly anisotropic diffusion tensors \cite{Dro14}.
However, due to its robustness and simplicity, linear TPFA is the method of choice in many engineering applications, including reservoir simulation, motivating the development of dedicated multilevel solvers.
\end{rmk}

\subsection{Single-level nonlinear solver}

Introducing a vector $\blkVec{x}$ collecting all the flux and pressure degrees of freedom, the nonlinear problem \eqref{eq:FV_matrix} written in residual form is
\begin{equation}
  \blkVec{r}( \blkVec{x} ) = 0,
  \label{eq:nonlinear_system_residual_form}
\end{equation}
where the residual is 
\begin{equation}
  \blkVec{r}( \blkVec{x} ) := \blkMat{A} ( \blkVec{x} ) - \blkVec{b},
  \label{eq:fine_level_residual}
\end{equation}
and 
\begin{equation}
   \blkMat{A}(\blkVec{x}) := 
   \begin{bmatrix}
    \Mat{M}(\Vec{p}) & \Mat{D}^T\\
    \Mat{D}      & 0
  \end{bmatrix} 
  \begin{bmatrix}
    \bsigma\\
    \blkVec{p}
  \end{bmatrix}
    \quad \text{and} \quad
    \blkVec{b} := 
      -
  \begin{bmatrix}
    \Vec{g}\\
    \Vec{f}
  \end{bmatrix}.
  \label{eq:fine_level_operators}  
\end{equation}
Starting from an initial guess $\blkVec{x}_0$, a widely used strategy to solve \eqref{eq:nonlinear_system_residual_form} relies on a global, single-level nonlinear solver---typically Picard's or Newton's method---to iteratively drive the residual to below a user-defined tolerance.
At each nonlinear iteration, a system obtained by linearizing \eqref{eq:nonlinear_system_residual_form} about a current nonlinear iterate $\blkVec{x}_k$ is solved---with an iterative linear solver---to compute a solution update, $\Delta \blkVec{x}_k$.
This update is then applied to the solution iterate at nonlinear iteration $k$ to update
\begin{equation}
  \blkVec{x}_{k+1} := \blkVec{x}_k + s_k \Delta \blkVec{x}_k,
  \label{eq:solution_update}
\end{equation}
where $s_k\in (0,1]$ is a suitably chosen step length.
The permeability functions considered in this work---see for instance \eqref{eq:permeability_pressure_relationship}---introduce severe nonlinearities in \eqref{eq:nonlinear_system_residual_form} that can undermine the convergence of the nonlinear solver.
For such a problem, it is standard practice to employ a globalization technique to improve the radius of convergence.
The simple backtracking line search procedure summarized in Algorithm \ref{alg:backtracking} is one such technique.
It uses the step length $s_k$ in \eqref{eq:solution_update} as a scaling factor for the Newton search direction to avoid overshoots---i.e., solution updates that fail to decrease the global residual and often cause convergence failure.

\begin{algorithm}[t]
\caption{Backtracking line search}
\begin{algorithmic}[1]
\Function{\tt Backtracking}{$\blkVec{x}_k, \Delta \blkVec{x}_k, n_{\max}, \theta$}
        \State $n = 0$
        \State $s_k = 1$
	\While{$\blkVec{r}( \blkVec{x}_k + s_k \Delta \blkVec{x}_k) > \blkVec{r}( \blkVec{x}_k)$ {\bf and} $n < n_{\max}$}
	        \If{$\blkVec{r}( \blkVec{x}_k +  \frac12 s_k \Delta \blkVec{x}_k) > \theta \blkVec{r}( \blkVec{x}_k + s_k \Delta \blkVec{x}_k)$}
	            \State {\bf break}
	        \EndIf
		\State $s_k \leftarrow \frac12 s_k $
		\State $n \leftarrow n + 1$
	\EndWhile
	\State \Return{$\blkVec{x}_{k+1}  := \blkVec{x}_k + s_k \Delta \blkVec{x}_k$}
 \EndFunction
 \end{algorithmic}
 \label{alg:backtracking}
\end{algorithm}

\begin{rmk}
In the parameter list of Algorithm~\ref{alg:backtracking}, $n_{\max}$ designates the maximum number of backtracking steps.
The scalar $\theta \in (0,1]$ is a threshold such that if the relative residual reduction between two consecutive backtracking steps is greater than $\theta$, the algorithm terminates and returns the current backtracking step.
\end{rmk}

\begin{rmk}
\label{rmk:special_backtracking}
A backtracking algorithm designed specifically to tackle the nonlinearity arising from the permeability function \eqref{eq:permeability_pressure_relationship} has also been implemented and is presented in Section~\ref{sec:numerics}.
\end{rmk}

For single-level approaches, scalability is achieved by designing an efficient preconditioner for the linearized systems to accelerate the convergence of the linear solver.
For linearized systems arising from the discretization of elliptic PDEs, Algebraic MultiGrid (AMG) is frequently a key component of the linear solution strategy, particularly when using unstructured grids. 
The nonlinear solver proposed in this work fundamentally differs from the standard approach reviewed here.
Instead of linearizing the problem first, and then possibly applying a multilevel solver to the linearized system, we directly apply a multilevel solution algorithm to the nonlinear system.

\section{Nonlinear multigrid}\label{sec:multigrid}

We propose a multilevel solution strategy for the discrete nonlinear system \eqref{eq:FV_matrix} based on the Full Approximation Scheme (FAS) \cite{brandt77,henson2003multigrid}.
At each iteration of the nonlinear solver, the solution update is computed recursively by solving nonlinear problems constructed on a hierarchy of nested approximation spaces, referred to as levels.
We use the convention that level $\ell = 0$ refers to the finest level (i.e., the original problem), and a larger value of $\ell$ means a coarse level.
%
%
%
This multigrid algorithm aims at improving the scalability and robustness of the nonlinear solver compared to the standard, single-level approach described in the previous section.

We first introduce some notation for the nonlinear problem solved on each level.
The superscript ``0'' denotes operators on the fine level.
Considering level $\ell$ and nonlinear iteration $k$, let $\blkVec{x}^{\ell}_k$ be a solution iterate, $\blkMat{A}^{\ell}(\blkVec{x}^{\ell}_k)$ be the coarse-level approximation of the fine-level operator $\blkMat{A}^0( \blkVec{x}^0_k)$ introduced in \eqref{eq:fine_level_operators}, and $\blkVec{b}^{\ell}_k$ be the approximation of the fine-level right-hand side $\blkVec{b}^0$. 
In residual form, the nonlinear problem solved on level $\ell$ at iteration $k$ reads
\begin{equation}
\blkVec{r}^{\ell}_k( \blkVec{x}^{\ell}_k ) := \blkMat{A}^{\ell}(\blkVec{x}_k^\ell) - \blkVec{b}^\ell_k = 0.
\label{eq:nonlinear_system_level_l}
\end{equation}
Algorithm~\ref{alg:fas} presents the key steps of a V-cycle FAS multigrid solver.  Note that the algorithm is recursively defined, so we work with generic levels $\ell$ and $\ell+1$.
Several building blocks are needed to define a specific solver for our nonlinear diffusion problem.  In particular, we need a nonlinear smoothing algorithm, three intergrid transfer operators---namely, an interpolation operator $\interpolate$, a restriction operator $\restrict$, and a projection operator $\inject$---and a strategy to construct the coarse problem operator $\blkMat{A}^{\ell+1}(\blkVec{x}_k^{\ell+1})$.  We begin with a high-level overview of these components, before presenting each in detail in subsequent sections.

In a nonlinear pre-smoothing step, we first apply a global linearization method, using either Picard's or Newton's method. Section~\ref{sec:smoothing} describes the construction of the required linearized operators at level $\ell$.
The linearized system is then solved to compute an update to the solution iterate, using a well-suited linear solver strategy (Section~\ref{sec:solvers-for-linearized-problems}).
Note that the pre-smoothing step includes a backtracking line search to avoid overshoots after the update.

After the smoothing step, we proceed to the algebraic construction of the coarse problem at level $\ell+1$.
This step requires the definition of a residual restriction operator from level $\ell$ to level $\ell+1$, denoted by $\restrict$, as well as the definition of a projection operator $\inject$ taking the level $\ell$ iterate to the level $\ell+1$ initial iterate.
In this work, $\restrict$ is taken as the transpose of the prolongation operator, $\interpolate$, i.e., $\restrict := (\interpolate)^T$.
%
The projection and prolongation operators are constructed such that they satisfy $\inject \interpolate = \blkMat{I}^{\ell}$.
%
%
These operators are computed in a pre-processing step using a methodology based on local spectral coarsening \cite{ml-spectral-coarsening} reviewed in Section~\ref{sec:coarsening}.
Using the definition of the residual given in \eqref{eq:nonlinear_system_level_l}, the coarse nonlinear problem at level $\ell+1$ is obtained in residual form as
\begin{equation}
  \blkVec{r}^{\ell+1}_{k}( \blkVec{y}^{\ell+1}_k ) := \blkMat{A}^{\ell+1}( \underbrace{\inject \blkVec{x}_k^\ell + \blkVec{e}^{\ell+1}_k}_{\blkVec{y}_k^{\ell+1}}) - \underbrace{( \blkMat{A}^{\ell+1}( \inject \blkVec{x}_k^\ell ) - \restrict \blkVec{r}^{\ell}_k( \blkVec{x}^{\ell}_k ) )}_{\blkVec{b}^{\ell+1}_k} = 0,
  \label{eq:coarse_level_residual}
\end{equation}
where the correction to the solution iterate at level $\ell + 1$ is the difference between the solution of \eqref{eq:coarse_level_residual}, denoted by $\blkVec{y}^{\ell+1}_{k}$, and the projection of the solution iterate at level $\ell$:
\begin{equation}
  \blkVec{e}^{\ell+1}_k := \blkVec{y}_k^{\ell+1} - \inject \blkVec{x}_k^{\ell}.
\end{equation}
The coarse operator $\blkMat{A}^{\ell+1}$ can be efficiently obtained using the methodology presented in Section~\ref{sec:all-level-M}.
The coarse correction $\blkVec{e}^{\ell+1}_k$  is interpolated to the fine level as $\interpolate \blkVec{e}^{\ell+1}_k$, and it is then used to increment the current approximation of the solution at level $\ell$.
A backtracking line search is applied to this step to avoid an overshoot after the application of the interpolated coarse correction.
A nonlinear post-smoothing step---that also includes a backtracking line search---concludes the computations at this level.
%
%
\begin{algorithm}[t]
\caption{Nonlinear step at level $\ell$ in the Full Approximation Scheme}
\label{alg:fas}
\begin{algorithmic}[1]
\Function{\tt NonlinearMG}{$\ell, \blkVec{x}^\ell_k, \blkVec{b}_k^\ell$}
	\If{$\ell$ is the coarsest level}
		\State Solve $\blkMat{A}^\ell(\blkVec{x}_k^\ell) - \blkVec{b}_k^\ell = 0$ for $\blkVec{x}_k^\ell$
	\Else
		\State $\blkVec{x}_k^\ell \leftarrow$ \Call{\tt NonlinearSmoothing}{$\ell, \blkVec{x}_k^\ell, \blkVec{b}_k^\ell$} \label{alg:line:nonlinPreSmoothing}
		\State $\blkVec{x}^{\ell+1}_k \leftarrow \inject \blkVec{x}_k^\ell$
		\State $\blkVec{b}_k^{\ell+1} \leftarrow \blkMat{A}^{\ell+1}(\blkVec{x}_k^{\ell+1}) - \restrict(\blkMat{A}^\ell(\blkVec{x}_k^\ell) - \blkVec{b}_k^\ell )$ 
		\State $\blkVec{y}_k^{\ell+1} \leftarrow$ \Call{\tt NonlinearMG}{$\ell+1, \blkVec{x}_k^{\ell+1}, \blkVec{b}_k^{\ell+1}$}
		\State $\blkVec{x}_k^\ell \leftarrow$ \Call{\tt Backtracking}{$\blkVec{x}_k^\ell, \interpolate( \blkVec{y}_k^{\ell+1} - \blkVec{x}_k^{\ell+1}), n_{\max}, \theta$} 
		\State $\blkVec{x}_k^\ell \leftarrow$ \Call{\tt NonlinearSmoothing}{$\ell, \blkVec{x}_k^\ell, \blkVec{b}_k^\ell$} \label{alg:line:nonlinPostSmoothing}
  	\EndIf
	\State \Return{$\blkVec{x}_k^\ell$}
 \EndFunction
 \end{algorithmic}
\end{algorithm}
We are now in a position to review the key components of the nonlinear multigrid solver, starting from the construction of the projection 
 and prolongation operators, $\inject$ and $\interpolate$.

\subsection{Construction of the coarse approximation spaces}\label{sec:coarsening}

The construction of the coarse approximation spaces is a key component of the nonlinear multigrid solver developed in this work. 
%
%
This pre-processing step is achieved by applying the multilevel spectral coarsening method for graph-Laplacians introduced in \cite{ml-spectral-coarsening}.
We underline that this coarsening method exploits the saddle-point structure of \eqref{eq:fine_level_residual}-\eqref{eq:fine_level_operators} and is the main motivation for considering the problem in mixed form.

The starting point of the methodology is the definition of a hierarchy of coarse triangulations in which the coarse cells are obtained by agglomerating fine cells.
Using the fine triangulation $\mathcal{T}$ introduced in Section~\ref{sec:FV_tpfa}, we construct an undirected graph, $G_{\mathcal{T}}$, whose vertices represent the cells in the fine triangulation.
In $G_{\mathcal{T}}$, two vertices are connected by an edge if the corresponding cells share an interface.
For a given discrete pressure, $\Vec{p}$, we note that
\begin{equation}
  \Mat{L} := \Mat{D}^0 (\Mat{M}^0 (\Vec{p}))^{-1}( \Mat{D}^0)^T
\end{equation}
is a weighted graph-Laplacian on $G_{\mathcal{T}}$ as defined in \cite{ml-spectral-coarsening}.

A graph partitioner (METIS) \cite{karypis1998fast} is then employed to partition $G_{\mathcal{T}}$ into connected, non-overlapping subsets of vertices defining aggregates of fine cells by duality.
After the partitioning, the cell aggregates are post-processed to ensure that a coarse interface between two aggregates is also a connected subset of fine interfaces.
This methodology, illustrated in Fig.~\ref{fig:aggregates-mesh-graph} for an unstructured mesh, results in a coarse triangulation made of non-standard coarse cells with non-planar interfaces.
The cell agglomeration procedure is applied recursively to obtain a hierarchy of coarse meshes.

The coarse triangulations are then used to build the nested approximation spaces.
We construct respectively the operator used to interpolate the coarse correction from level $\ell+1$ to $\ell$, denoted by $\interpolate$, and the operator that projects the approximate solution at level $\ell$ to $\ell+1$, denoted by $\inject$.
Since the discrete nonlinear system \eqref{eq:nonlinear_system_level_l} involves two types of unknowns---flux and pressure---we look for $\interpolate$ and $\inject$ in a $2\times2$ block-diagonal form
\begin{equation}
\interpolate = \begin{bmatrix}
(\Mat{P}_\sigma)_{\ell+1}^\ell & \\ & (\Mat{P}_p)_{\ell+1}^\ell
\end{bmatrix} \qquad \text{ and } \qquad \inject = \begin{bmatrix}
(\Mat{Q}_\sigma)_\ell^{\ell+1} & \\ & (\Mat{Q}_p)_\ell^{\ell+1}
\end{bmatrix},
\label{eq:prolongation_and_projection_operators}
\end{equation}
such that the columns of $\interpolate$ are linearly independent and $\inject\interpolate = \blkMat{I}^{\ell+1}$.
The construction of the operators \eqref{eq:prolongation_and_projection_operators}, reviewed in \ref{sec:construction_of_prolongation_and_projection}, only involves local computations.
In brief, the definition of the coarse vertex-based (pressure) degrees of freedom is first performed aggregate-by-aggregate using the low-energy eigenvectors of the local graph-Laplacian operator.
This step yields the prolongation operator $(\Mat{P}_p)_{\ell+1}^\ell$.
Then, we introduce the coarse edge-based (flux) degrees of freedom that are used to obtain $(\Mat{P}_\sigma)_{\ell+1}^\ell $ and a locally constructed projection $\inject$ that satisfies the commutativity property $(\Mat{D}^{\ell+1}) (\Mat{Q}_{\sigma})^{\ell+1}_{\ell} = (\Mat{Q}_p)^{\ell+1}_{\ell} (\Mat{D}^{\ell})$.
The matrix $\Mat{D}^{\ell+1}$ is defined variationally as the product $\left ((\Mat{P}_p)_{\ell+1}^{\ell}\right )^T \Mat{D}^{\ell} (\Mat{P}_\sigma)^\ell_{\ell+1}$.
This procedure can then be applied recursively to obtain nested approximation spaces with an increasing degree of coarsening.

\begin{rmk}[Abuse of terminology]

On level $\ell$, owing to the condition $\inject\interpolate = \blkMat{I}^{\ell+1}$, it is obvious that the operator $\blkMat{\Pi}^\ell:=\interpolate \inject$ is a projection because
\begin{equation}
(\blkMat{\Pi}^\ell)^2 = \blkMat{\Pi}^\ell.
\label{eq:projection_property}
\end{equation}
Since the columns of $\interpolate$ are linearly independent, the projection $\blkMat{\Pi}^\ell \blkVec{x}^\ell = \interpolate \inject\blkVec{x}^\ell$ of any vector $\blkVec{x}^\ell$ on level $\ell$ can be uniquely identified by the coefficient vector $\inject\blkVec{x}^\ell$ on level $\ell+1$. In essence, $\blkMat{\Pi}^\ell \blkVec{x}^\ell$ and $\inject\blkVec{x}^\ell$ are the representations of the coarse-space projection of $\blkVec{x}^\ell$ on level $\ell$ and $\ell+1$ respectively.
Hence, although $\inject$ does not satisfy the formal definition of a projection \eqref{eq:projection_property}, with an abuse of terminology, $\inject$ is also referred to as a ``projection" to the coarse space throughout the paper.

\end{rmk}

\begin{figure}
  \hfill
  \begin{subfigure}[c]{.35\linewidth}
    \centering
    \includegraphics[width=\linewidth]{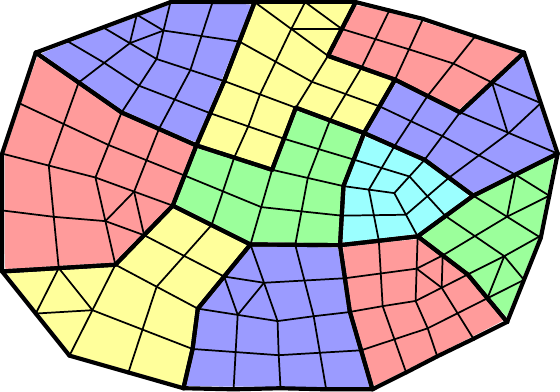}
    \caption{}
  \end{subfigure}
  \hfill
  \begin{subfigure}[c]{.35\linewidth}
    \centering
    \includegraphics[width=\linewidth]{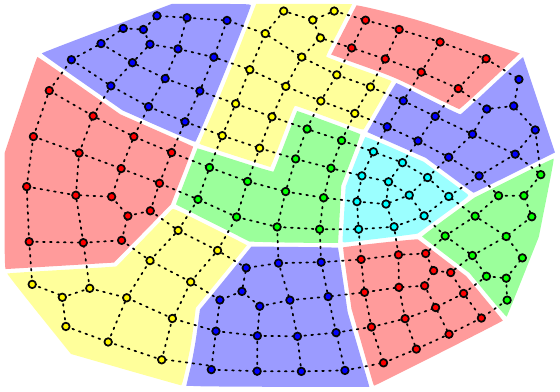}
    \caption{}
  \end{subfigure}
  \hfill\null
  \caption{\label{fig:aggregates-mesh-graph}Illustration of the partitioning methodology used to define the coarse space.
First, we form a graph whose vertices and edges correspond, respectively, to the cells and interfaces in the mesh.
We partition the graph into connected subsets of vertices defining, by duality, cell aggregates with non-planar interfaces.
Then, we use a local algebraic procedure based on the subsets of vertices to construct the coarse vertex-based and edge-based approximation spaces.
}
\end{figure}

\subsection{Formation of the discrete operator on each level}\label{sec:all-level-M}

Here, we discuss the construction of the discrete nonlinear operator $\blkMat{A}^{\ell}$ as a function of the solution iterate on the same level, $\blkVec{x}^{\ell}_k$.
This operator appears in the nonlinear problem \eqref{eq:nonlinear_system_level_l} and is written in the form
\begin{equation}
\blkMat{A}^{\ell}(\blkVec{x}^{\ell}_k) :=
 \begin{bmatrix}
    \Mat{M}^{\ell}(\widetilde{\Mat{\Pi}}^{\ell}\Vec{p}^{\ell}_k) & \left(\Mat{D}^{\ell}\right)^T\\
    \Mat{D}^{\ell} & 
  \end{bmatrix} \begin{bmatrix}
    \Vec{\sigma}^{\ell}_k \\ \Vec{p}^{\ell}_k
  \end{bmatrix},
  \label{eq:structure}
\end{equation} 
where $\widetilde{\Mat{\Pi}}^\ell$ is a projection from the pressure space on level $\ell$ to the space of piecewise-constant functions on the same level, with each constant representing the average pressure value in a given aggregate on that level. 
The algebraic definition of $\widetilde{\Mat{\Pi}}^{\ell}$ is given in \ref{sec:piecewise-constant-projection}.
The reason to include $\widetilde{\Mat{\Pi}}^{\ell}$ in the definition of $\blkMat{A}^{\ell}(\blkVec{x}^{\ell}_k)$ is that it allows us to directly evaluate $\blkMat{A}^{\ell}(\blkVec{x}^{\ell}_k)$ without visiting the finest level during the multigrid cycle.
As a result, we can substantially reduce the complexity of one multigrid cycle, at the cost of sacrificing some accuracy because $\widetilde{\Mat{\Pi}}^{\ell}\Vec{p}^{\ell}_k$ is just an approximation to $\Vec{p}^{\ell}_k$.
As far as scalability is concerned, the reduction in complexity will outweigh the relative loss of accuracy.
We remind the reader that the components of the fine-level operator, $\blkMat{A}^{0}(\blkVec{x}^{0}_k)$, have already been defined in \eqref{eq:FV_matrix}.
%
%
Considering now a coarse level $\ell > 0$, the discrete nonlinear operator is conceptually obtained by applying a Galerkin projection on the finer level, $\ell-1$, in the sense that
\begin{equation}
 \begin{bmatrix}
    \Mat{M}^{\ell}(\widetilde{ \Mat{\Pi} }^{\ell}\Vec{p}^{\ell}_k) & \left(\Mat{D}^{\ell}\right)^T\\
    \Mat{D}^{\ell} & 
  \end{bmatrix} =
  \left(  \blkMat{P}_\ell^{\ell-1} \right)^T
  \begin{bmatrix}
    \Mat{M}^{\ell-1}(\widetilde{ \Mat{\Pi} }^\ell \Vec{p}^\ell_k ) & \left( \Mat{D}^{\ell-1}\right)^T\\
    \Mat{D}^{\ell-1} & 
  \end{bmatrix}
  \blkMat{P}_\ell^{\ell-1} .
  \label{eq:coarse-saddle}
\end{equation}
We note that \eqref{eq:coarse-saddle} does not provide a practical way to compute $\blkMat{A}^{\ell}$. 
Using \eqref{eq:coarse-saddle} to form $\blkMat{A}^{\ell}$ would require updating the finer-level operator $M^{\ell-1}$ with the coarse-level pressure iterate, $\Vec{p}^{\ell}_k$, at each iteration. 
If $\ell-1 \neq 0$, then this recursive updating procedure would be repeated until the finest level is reached. 
Therefore, the update of the discrete problems on all levels would involve fine-grid computations, which would significantly increase the cost of the algorithm.
For the multigrid solver to be efficient and scalable, it is crucial to minimize the calculations performed on the fine grid. 
Fortunately, $\Mat{M}^\ell(\widetilde{ \Mat{\Pi} }^\ell \Vec{p}^{\ell}_k)$ can be formed in a different way that does not require visiting the fine grid, as shown below. 
This is achieved by assembling the coarse system from coarse local transmissibility matrices that satisfy the decomposition given in the following proposition.

We adopt the indexing notation of \cite{ml-spectral-coarsening}, in which a cell at a coarse level, $\ell \geq 1$, corresponds to an aggregate, $\mathcal{A}^{\ell-1}$, formed by agglomerating finer cells at level $\ell-1$.
To simplify the notation, we omit the subscript $k$ denoting the nonlinear iteration in the remainder of this section. 
%
\begin{proposition}\label{prop:assemble}
On each level $\ell \geq 0$, $\Mat{M}^{\ell}( \widetilde{ \Mat{\Pi}}^{\ell}\Vec{p}^{\ell})$ can be assembled aggregate-by-aggregate from local transmissibility matrices.
Considering a coarse level $\ell \geq 1$, a local transmissibility matrix can be decomposed as
\begin{equation}
\Mat{M}_{\mathcal{A}^{\ell-1}}^{\ell}(\widetilde{ \Mat{\Pi} }^{\ell} \Vec{p}^{\ell}) = \frac{1}{\kappa(p^{\ell}_{\mathcal{A}^{\ell-1}})} \Mat{M}_{\mathcal{A}^{\ell-1}}^{\ell},  
\label{eq:local-product-decomposition}
\end{equation}
where $\mathcal{A}^{\ell-1}$ is an aggregate of finer cells at level $\ell-1$, or equivalently, a coarse cell at level $\ell$.
In \eqref{eq:local-product-decomposition}, $\Mat{M}_{\mathcal{A}^{\ell-1}}^{\ell}$ is a pre-computed matrix defined on $\mathcal{A}^{\ell-1}$ that is independent of pressure, and $p^{\ell}_{\mathcal{A}^{\ell-1}} := (\widetilde{ \Mat{\Pi}}^{\ell} \Vec{p}^{\ell})|_{\mathcal{A}^{\ell-1}}$ is the average pressure value on $\mathcal{A}^{\ell-1}$.
A decomposition analogous to \eqref{eq:local-product-decomposition} exists at the finest level for a local transmissibility matrix defined on a single cell $\tau_{K} \in \mathcal{T}$.
\end{proposition}

Before discussing the proof of Proposition~\ref{prop:assemble} for the coarse levels, we focus on the case of the finest level.
%
%
Let us consider a fine cell $\tau_K \in \mathcal{T}$.
We remind the reader that on the fine level, $\widetilde{ \Mat{\Pi} }^0$ is the identity map.
The fine local transmissibility matrix in $\tau_K$ is defined as a product between the scalar $1/\kappa(p^0_{\tau_K})$ and a diagonal matrix that is independent of the local pressure, $p^0_{\tau_K}$:
\begin{equation}
  \Mat{M}^0_{\tau_K}(p^0_{\tau_K}) :=
  \frac{1}{\kappa(p^0_{\tau_K})}
  \textbf{diag} ( \overline{\Vec{\Upsilon}}_{K,inv} ).
\label{eq:fine-local-product-decomposition}
\end{equation}
Here, $\textbf{diag}( \overline{\Vec{\Upsilon}}_{K,inv} )$ is the diagonal matrix created from the entries of the argument vector
\begin{equation}
  \overline{\Vec{\Upsilon}}_{K,inv} = \bigg( \overline{\Upsilon}_{K,\varepsilon}^{\;-1} \bigg)_{\varepsilon \in \partial \tau_K}
\end{equation}
collecting the reciprocal of the geometric term of the one-sided transmissibility \eqref{eq:half-transmissibility} associated with interfaces belonging to $\tau_K$.
Now, let $\widehat{\Mat{M}}^0(\Vec{p}^0)$ be the block-diagonal matrix whose $K$th diagonal block is the local transmissibility matrix, $\Mat{M}^0_{\tau_K}(p^0_{\tau_K})$.
To assemble the global matrix, we introduce the Boolean rectangular sparse matrix $W_{\hat{\sigma}\sigma}^0$,   which maps global flux degrees of freedom into local ones \cite{Vas08}.
Then, the global matrix  $M^0(\Vec{p}^0)$ is obtained by assembling the local matrices as:
\begin{equation}
\Mat{M}^0(\Vec{p}) = (W_{\hat{\sigma}\sigma}^0)^T \widehat{\Mat{M}}^0(\Vec{p}^0) W_{\hat{\sigma}\sigma}^0.
\label{eq:assemble}
\end{equation}
This concludes the proof of Proposition~\ref{prop:assemble} on the fine level.

We prove Proposition~\ref{prop:assemble} for the coarse levels by induction on $\ell$.
Let us consider a coarse cell at level $\ell+1$ corresponding to a cell aggregate $\mathcal{A}^{\ell}$.
Briefly speaking, the key idea of the proof is to first use the induction hypothesis at level $\ell$ to construct the local transmissibility matrices in the finer cells of level $\ell$ that have been agglomerated to form $\mathcal{A}^{\ell}$.
After that, we coarsen locally to prove that the decomposition \eqref{eq:local-product-decomposition} can be obtained in the coarse cell at level $\ell+1$.
We note that this decomposition is possible because $\widetilde{\Mat{\Pi}}^{\ell+1}\Vec{p}^{\ell+1}$ is constant on $\mathcal{A}^{\ell}$.
Then, the global-to-local map for the flux degrees of freedom on level $\ell+1$, denoted by $W_{\hat{\sigma}\sigma}^{\ell+1}$, is used to assemble the global system.
The details of the proof are in \ref{sec:assemble}.

\begin{rmk}
Proposition~\ref{prop:assemble} plays a central role in the assembly of $\blkMat{A}^{\ell}$ on the coarse levels.
The local, static matrices $\Mat{M}^{\ell}_{\mathcal{A}^{\ell-1}}$ are pre-computed and stored when the multigrid hierarchy is constructed during the setup phase.
Then, at each nonlinear iteration of the solving phase, $\Mat{M}^{\ell}_{\mathcal{A}^{\ell-1}}( \widetilde{ \Mat{\Pi} }^{\ell}\Vec{p}^{\ell})$ is formed using \eqref{eq:local-product-decomposition}.
That is, we scale $\Mat{M}^{\ell}_{\mathcal{A}^{\ell-1}}$ with $1/\kappa( p^{\ell}_{\mathcal{A}^{\ell-1}})$ in each cell aggregate to construct $\Mat{M}^{\ell}_{\mathcal{A}^{\ell-1}}( \widetilde{ \Mat{\Pi} }^{\ell} \Vec{p}^{\ell}  )$, and then we assemble the local contribution into the global system using $W_{\hat{\sigma}\sigma}^{\ell}$.
With this procedure, the formation of $\Mat{M}^{\ell}( \widetilde{ \Mat{\Pi} }^{\ell} \Vec{p}^{\ell} )$ only involves operations on level $\ell$.
\end{rmk}


\subsection{Nonlinear smoothing}\label{sec:smoothing}

In Algorithm~\ref{alg:fas}, the nonlinear smoothing step (lines \ref{alg:line:nonlinPreSmoothing} and \ref{alg:line:nonlinPostSmoothing}), consists of the application of either one Picard or one Newton iteration to approximate the solution of the nonlinear problem \eqref{eq:nonlinear_system_level_l} on level $\ell$. 
Both approaches require linearizing \eqref{eq:nonlinear_system_level_l} at $\blkVec{x}^{\ell}_k$, solving the resulting linear system to compute a solution update, and then applying the backtracking line search to this update.
Since the discrete nonlinear operators on all levels share the structure of \eqref{eq:structure}, the smoothing step on all levels can be described in a unified way. For this reason, we drop the superscript $\ell$ on the operators in the following description of the smoothers.

\subsubsection{Picard iteration}

The Picard smoothing step starts with the assembly of the matrix
$\Mat{M}(\widetilde{\Mat{\Pi}} \Vec{p}_{k} )$ using the argument vector, $\blkVec{x}_k$.
This is done using the methodology described in Section~\ref{sec:all-level-M}.
Then, since the derivatives of the entries of $\Mat{M}( \widetilde{\Mat{\Pi} }\Vec{p}_{k} )$ with respect to pressure are neglected in the Picard linearization, we solve the linear system
\begin{equation}
  \begin{bmatrix}
    \Mat{M}(\widetilde{ \Mat{\Pi} }\Vec{p}_{k}) & \Mat{D}^T\\
    \Mat{D} & 0
  \end{bmatrix} 
  \begin{bmatrix}
    \Vec{\sigma}_P\\
    \Vec{p}_P
  \end{bmatrix} 
  = -
  \begin{bmatrix}
    \Vec{g}_k \\
    \Vec{f}_k
  \end{bmatrix},
  \label{eq:picardian_system}
\end{equation}
for the Picard solution, $\blkVec{x}_P := \begin{bmatrix} \Vec{\sigma}_P \\ \Vec{p}_P \end{bmatrix}$.
To conclude the smoothing step, we apply the backtracking line search of Algorithm~\ref{alg:backtracking} so that the Picard-based smoothing step returns $\blkVec{x}_k + s_k ( \blkVec{x}_P - \blkVec{x}_k )$.

\subsubsection{Newton iteration}

The Newton linearization at $\blkVec{x}_k$ yields the system
\begin{equation}
\blkMat{J}_k (\blkVec{x}_k) \Delta \blkVec{x}_k = - \blkVec{r}_k( \blkVec{x}_k ). \label{eq:jacobian_system}
\end{equation}
We solve \eqref{eq:jacobian_system} for the Newton update $\Delta \blkVec{x}_k := \begin{bmatrix}  \Delta \blkVec{\sigma}_k \\ \Delta \blkVec{p}_k \end{bmatrix}$. 
The Jacobian matrix, $\blkMat{J}_k$, is defined as
\begin{equation}
  \blkMat{J}_k (\blkVec{x}_k) := \partial_{\blkVec{x}} \blkVec{r}_k (\blkVec{x}_k) =
  \begin{bmatrix}
   \partial_{\Vec{\sigma}} (\Vec{r}_{\Vec{\sigma}})_k ( \Vec{\sigma}_{k}, \Vec{p}_{k}) & 
   \partial_{\Vec{p}} (\Vec{r}_{\Vec{\sigma}})_k ( \Vec{\sigma}_{k}, \Vec{p}_{k}) \\
   \partial_{\Vec{\sigma}} (\Vec{r}_{\Vec{p}})_k ( \Vec{\sigma}_{k}, \Vec{p}_{k}) &
   \partial_{\Vec{p}} (\Vec{r}_{\Vec{p}})_k ( \Vec{\sigma}_{k}, \Vec{p}_{k})
  \end{bmatrix}.
  \label{eq:jacobian_matrix}
\end{equation}
It is easy to see that 
\begin{equation}
  \partial_{\Vec{\sigma}} (\blkVec{r}_{\Vec{\sigma}})_k ( \Vec{\sigma}_{k}, \Vec{p}_{k}) =  \Mat{M} (\widetilde{\Mat{\Pi}}\Vec{p}_{k}), \;\; \partial_{\Vec{\sigma}} (\blkVec{r}_{\Vec{p}})_k ( \Vec{\sigma}_{k}, \Vec{p}_{k}) = D, \;\; \text{and} \;\; \partial_{\Vec{p}} (\blkVec{r}_{\Vec{p}})_k ( \Vec{\sigma}_{k}, \Vec{p}_{k}) = 0,
  \label{eq:easy_jacobian_matrix_terms}
\end{equation}
where $\Mat{M}( \widetilde{\Mat{\Pi}}\Vec{p}_k )$ is assembled as explained in Section~\ref{sec:all-level-M}.
The evaluation of $\partial_{\Vec{p}} (\Vec{r}_{\Vec{\sigma}})_k ( \Vec{\sigma}_{k}, \Vec{p}_{k})$ requires differentiating
\begin{equation}
  ( \Vec{r}_{\Vec{\sigma}} )_k = \Mat{M}( \widetilde{\Mat{\Pi}}\Vec{p}_k ) \Vec{\sigma}_k + \Mat{D}^T \Vec{p}_k + \Vec{g}_k
  \label{eq:flux_residual}
\end{equation}
with respect to $\Vec{p}_k$. 
In particular, we want to extend the methodology developed in Section~\ref{sec:all-level-M} to evaluate the derivatives of the first term in the right-hand side, $\Mat{M}( \widetilde{\Mat{\Pi}}\Vec{p}_k ) \Vec{\sigma}_k$, without performing operations on the finer levels.
Before demonstrating in Proposition~\ref{prop:Mpsigma} how this can be achieved, we introduce some useful notation.
We first introduce the vector $\Vec{\kappa}_{inv}( \widetilde{\Mat{\Pi}} \Vec{p}_k )$ containing as many entries as the number of cells on the level.
These entries are defined as
\begin{equation}
  [ \Vec{\kappa}_{inv}( \widetilde{\Mat{\Pi}}\Vec{p}_k ) ]_i
  := \frac{1}{\kappa( (\widetilde{\Mat{\Pi}} \Vec{p}_k )|_{\mathcal{A}_i} ) }.
  \label{eq:definition_inverse_kappa_vector}
\end{equation}
We remind the reader that $\widehat{\Mat{M}}$ is the block-diagonal matrix introduced in Proposition~\ref{prop:assemble} whose $i$th block is the static matrix $\Mat{M}_{\mathcal{A}_i}$ defined on cell aggregate $\mathcal{A}_i$.
%
We can now discuss the following proposition stating that $\Mat{M}( \widetilde{ \Mat{\Pi} } \Vec{p}_k)\Vec{\sigma}_k$ is linear with respect to $\Vec{\kappa}_{inv}(\widetilde{\Mat{\Pi}}\Vec{p}_k)$.


\begin{proposition}
\label{prop:Mpsigma}
The nonlinear term $\Mat{M} (\widetilde{ \Mat{\Pi} }\Vec{p}_k)\Vec{\sigma}_k$ appearing in \eqref{eq:flux_residual} can be decomposed into the matrix-vector product
\begin{equation}
\Mat{M}(\widetilde{ \Mat{\Pi} }\Vec{p}_k)\Vec{\sigma}_k = \Mat{N}(\Vec{\sigma}_k) \Vec{\kappa}_{inv}(\widetilde{ \Mat{\Pi} }\Vec{p}_k).
\label{eq:product-decomposition-to-differentiate}
\end{equation}
In \eqref{eq:product-decomposition-to-differentiate}, the flux-dependent matrix, $\Mat{N}(\Vec{\sigma}_k)$, is defined as
\begin{equation}
  \Mat{N}(\Vec{\sigma}_k) := \Mat{W}_{\hat{\sigma}\sigma}^T \widehat{\Mat{M}} \diag( \Mat{W}_{\hat{\sigma}\sigma} \Vec{\sigma}_k ) \Mat{W}_{\hat{\sigma} p}.
\label{eq:definition-N}
\end{equation}
where the Boolean matrices $\Mat{W}_{\hat{\sigma}\sigma}$ and $\Mat{W}_{\hat{\sigma}p}$ represent, respectively, the map from global flux degrees of freedom to local flux degrees of freedom introduced in the proof of Proposition~\ref{prop:assemble}, and the map from cells/aggregates to local flux degrees of freedom.
\end{proposition}

The proof of Proposition~\ref{prop:Mpsigma} is in \ref{sec:Mpsigma}.
We remark that, unlike \eqref{eq:local-product-decomposition}, the decomposition \eqref{eq:product-decomposition-to-differentiate} is not directly used in the assembly of the Jacobian matrix.
Instead, the purpose of \eqref{eq:product-decomposition-to-differentiate} is to provide a practical way to differentiate $\Mat{M}(\widetilde{ \Mat{\Pi} }\Vec{p}_k)\Vec{\sigma}_k$ with respect to $\Vec{p}_k$. 
Next, we exploit the fact that $\Mat{M}(\widetilde{ \Mat{\Pi} }\Vec{p}_k)\Vec{\sigma}_k$ is linear with respect to $\Vec{\kappa}_{inv}(\widetilde{ \Mat{\Pi} }\Vec{p}_k)$ to obtain an expression for $\partial_{\Vec{p}} ( \Vec{r}_{\Vec{\sigma}} )_k (\Vec{\sigma}_{k}, \Vec{p}_{k})$.
\begin{proposition}\label{prop:d_flux_residual_d_pk}
  Consider the discrete nonlinear residual, $(\Vec{r}_{\Vec{\sigma}})_k (\Vec{\sigma}_k, \Vec{p}_k)$, defined in \eqref{eq:flux_residual}. Differentiating $(\Vec{r}_{\Vec{\sigma}})_k$ with respect to $\Vec{p}_k$ yields
\begin{equation}
\partial_{\Vec{p}} ( \Vec{r}_{\Vec{\sigma}} )_k (\Vec{\sigma}_{k}, \Vec{p}_{k}) = \Mat{N}(\Vec{\sigma}_{k}) \left( \frac{ d (\Vec{\kappa}_{inv}(\Vec{p})) }{d\Vec{p}}|_{\Vec{p} = \widetilde{\Mat{\Pi}}\Vec{p}_{k}}\right)\widetilde{ \Mat{\Pi} } + \Mat{D}^T,
  \label{eq:d_flux_residual_d_pk}
\end{equation}
where $\Vec{\kappa}_{inv}(\widetilde{\Mat{\Pi}}\Vec{p}_{k})$ and $\Mat{N}(\Vec{\sigma}_k)$ are defined, respectively, in \eqref{eq:definition_inverse_kappa_vector} and \eqref{eq:definition-N}.
\end{proposition}
\begin{proof}
To show that \eqref{eq:d_flux_residual_d_pk} holds, we first write that, by definition, $\partial_{\Vec{p}} ( \Vec{r}_{\Vec{\sigma}} )_k (\Vec{\sigma}_{k}, \Vec{p}_{k}) \Delta \Vec{p}_k$ is
\begin{equation}
\partial_{\Vec{p}} ( \Vec{r}_{\Vec{\sigma}} )_k (\Vec{\sigma}_{k}, \Vec{p}_{k}) \Delta \Vec{p}_k
=
\lim_{\epsilon\to 0} \frac{1}{\epsilon} \big( M(\widetilde{ \Mat{\Pi} }(\Vec{p}_{k}+\epsilon \Delta \Vec{p}_k))\Vec{\sigma}_{k}-M(\widetilde{ \Mat{\Pi} }\Vec{p}_{k})\Vec{\sigma}_{k} \big) + \Mat{D}^T\Delta \Vec{p}_k.
\end{equation}
We now use the decomposition \eqref{eq:product-decomposition-to-differentiate} from Proposition~\ref{prop:Mpsigma}, and the fact that $\Mat{N}(\Vec{\sigma}_k)$ is independent of pressure, to obtain
\begin{equation}
\partial_{\Vec{p}} ( \Vec{r}_{\Vec{\sigma}} )_k (\Vec{\sigma}_{k}, \Vec{p}_{k}) \Delta \Vec{p}_k
=
N( \Vec{\sigma}_{k}) \lim_{\epsilon\to 0} \frac{1}{\epsilon} \big( \Vec{\kappa}_{inv} (\widetilde{ \Mat{\Pi} } \Vec{p}_{k} + \epsilon \widetilde{ \Mat{\Pi} }\Delta \Vec{p}_k) - \Vec{\kappa}_{inv}(\widetilde{ \Mat{\Pi} }\Vec{p}_{k}) \big) + D^T \Delta \Vec{p}_k,
\label{eq:last_equation_proof_d_flux_residual_d_pk}
\end{equation}
where we have used the fact that $\widetilde{ \Mat{\Pi} }$ is linear.
Since the limit of the vector function in the right-hand side of \eqref{eq:last_equation_proof_d_flux_residual_d_pk} is equal, by definition, to the Jacobian of $\Vec{\kappa}_{inv}( \cdot )$ evaluated at $\widetilde{ \Mat{\Pi} }(\Vec{p}_k)$, we obtain \eqref{eq:d_flux_residual_d_pk}.
\end{proof}

This completes the description of the Jacobian assembly.
To summarize the Newton iteration, we use \eqref{eq:easy_jacobian_matrix_terms} and \eqref{eq:d_flux_residual_d_pk} to construct the four sub-blocks of $\blkMat{J}_k$.
Then, we solve the Jacobian linear system \eqref{eq:jacobian_system} for $\Delta \Vec{x}_k$.
After that, we apply the backtracking line search procedure of Algorithm \ref{alg:backtracking} to this Newton update.
The Newton-type smoothing step then returns $\blkVec{x}_k + s_k \Delta \blkVec{x}_k$.

\subsection{Solvers for the linearized problems}\label{sec:solvers-for-linearized-problems}

For both the Picard and Newton linearizations, the linear systems \eqref{eq:picardian_system} and \eqref{eq:jacobian_system} may be written in the form
\begin{equation}
  \begin{bmatrix}
    \Mat{M} & \Mat{E}\\
    \Mat{D} & 0
  \end{bmatrix}
  \begin{bmatrix}
    \Vec{\sigma} \\
    \Vec{p}
  \end{bmatrix}
  = -
  \begin{bmatrix}
    \Vec{g}\\
    \Vec{f}
  \end{bmatrix}.
  \label{eq:general-block-system}
\end{equation}
Here, $\Mat{E} := \Mat{D}^T$ in the Picard iterations while $\Mat{E} := \partial_{\Vec{p}} \Vec{r}_\sigma$ in the Newton iterations.
We consider two approaches to solve the $2\times2$ block-system \eqref{eq:general-block-system}, namely block-diagonal preconditioning and algebraic hybridization.

In the first approach, we solve the indefinite block system directly by a Krylov-subspace iterative method accelerated by a block-diagonal preconditioner. 
In the Picard iterations, \eqref{eq:general-block-system} is solved by the minimal residual method (MINRES) because the Jacobian matrix is symmetric. 
The generalized minimal residual method (GMRES) is used in the Newton iterations since the Jacobian system arising in Newton's method is not symmetric. 
The block-diagonal preconditioner \citep{block-diagonal, benzi05} is defined as
\begin{equation}
 \mathcal{B} =
 \begin{bmatrix}
    \mathcal{B}_M & 0\\
    0 & \mathcal{B}_S
 \end{bmatrix}.
\end{equation}
Here, $\mathcal{B}_M$ is the Jacobi smoother for $M$, and $\mathcal{B}_S$ is the BoomerAMG preconditioner \cite{henson02, hypre} for the approximate Schur complement $\Mat{S} := \Mat{D} \left( \text{diagm}(\Mat{M})\right)^{-1}\Mat{E}$, where $\text{diagm}(\Mat{M})$ is the diagonal matrix formed by extracting the diagonal entries of $\Mat{M}$.

In the second approach, the block system is solved by an algebraic variant of hybridization \citep{dobrev18}.
Hybridization is a classic approach for solving saddle-point systems arising from mixed finite-element discretizations; see, for example, \cite{arnold85}. 
In hybridization methods, the  continuity assumption on the approximation space of the flux unknown is removed. 
Instead, flux continuity is enforced weakly through a set of constraint equations. 
In this work, we use an algebraic variant of hybridization because the fine and coarse discrete systems considered here do not come from a finite-element discretization. 
In compact form, the structure of the hybridized system is
\begin{equation}
  \begin{bmatrix}
    \widehat{\Mat{M}} & \widehat{\Mat{E}} & \Mat{C}^T\\
    \widehat{\Mat{D}} & 0 & 0\\
    \Mat{C} & 0 & 0
  \end{bmatrix} 
  \begin{bmatrix}
    \widehat{\Vec{\sigma}} \\
    \Vec{p} \\
    \Vec{\lambda}
  \end{bmatrix} 
  = -
  \begin{bmatrix}
    \widehat{\Vec{g}}\\
    \Vec{f} \\
    0
  \end{bmatrix}.
  \label{eq:hybridized-block-system}
\end{equation}
Here, $\widehat{\Mat{M}}$, $\widehat{\Mat{D}}$, and $\widehat{\Mat{E}}$ are block-diagonal matrices because the flux continuity assumption has been removed.
The blocks in $\widehat{\Mat{M}}, \widehat{\Mat{D}}$, and $\widehat{\Mat{E}}$ are precisely the local matrices of $\Mat{M}$, $\Mat{D}$, and $\Mat{E}$, respectively.
The local matrices of $\Mat{M}$ and $\Mat{E}$ have been presented in Sections~\ref{sec:all-level-M} and \ref{sec:smoothing}.
In practice, instead of solving \eqref{eq:hybridized-block-system} directly, we consider the equivalent reduced system
\begin{equation}
\Mat{H} \Vec{\lambda} = \Vec{\xi},
\label{eq:hybridized-system}
\end{equation}
where
\begin{equation}
\Mat{H} := \begin{bmatrix}
    \Mat{C} & 0
  \end{bmatrix}
  \begin{bmatrix}
    \widehat{\Mat{M}} & \widehat{\Mat{E}} \\
    \widehat{\Mat{D}} & 0
  \end{bmatrix} ^{-1}
  \begin{bmatrix}
     \Mat{C}^T \\ 0
    \end{bmatrix},
    \quad \text{and} \quad 
    \Vec{\xi} := - \begin{bmatrix}
    \Mat{C} & 0
  \end{bmatrix}
  \begin{bmatrix}
    \widehat{\Mat{M}} & \widehat{\Mat{E}}\\
    \widehat{\Mat{D}} & 0
  \end{bmatrix} ^{-1}
  \begin{bmatrix}
     \widehat{\Vec{g}} \\
     \Vec{f}
    \end{bmatrix}.
\end{equation}
It is advantageous to solve \eqref{eq:hybridized-system} because the global system $\Mat{H}$ has a smaller size than \eqref{eq:hybridized-block-system}. Once \eqref{eq:hybridized-system} is solved, the unknowns $\widehat{\Vec{\sigma}}$ and $\Vec{p}$ are then recovered by local back substitution.
In the Picard iterations, $\Mat{H}$ is symmetric positive definite, and therefore \eqref{eq:hybridized-system} can be solved by the conjugate gradient method (CG).
In the Newton iterations, $\Mat{H}$ is non-symmetric, so the system is solved by GMRES. 
As shown in \cite{dobrev18} from an algebraic perspective, $\Mat{H}$ has a small near-nullspace, which motivates the use of  AMG is a good candidate for an effective preconditioner for $\Mat{H}$. 
This suggestion is supported by the numerical tests in \cite{lee17} and \cite{dobrev18}. 
Therefore, for both the Picard and Newton iterations, we use AMG as a preconditioner for the corresponding Krylov solvers. 
In particular, as in the block-diagonal preconditioner, we choose BoomerAMG in the numerical examples.

We have observed in our numerical tests that block-diagonal preconditioning is faster than algebraic hybridization on the fine-level systems, especially for 3D problems. %
In contrast, for the coarse-level systems algebraic hybridization exhibits a much faster convergence rate than block-diagonal preconditioning. 
Consequently, we apply the block-diagonal preconditioning approach to the fine-level systems, and the algebraic hybridization approach to the coarse-level systems.

\section{Numerical examples}\label{sec:numerics}

To demonstrate the relative performance of the proposed nonlinear multigrid solver compared to single level methods, we apply these solution strategies to solve three benchmark examples.
Even though the fine grids of some of these examples are structured, the coarse aggregates are unstructured and we do not actually exploit any structured information during coarsening.
Instead, the aggregates are formed using METIS, cf. Section~\ref{sec:coarsening}.
In the following sections, we report a coarsening factor between levels computed using the average aggregate size on each level.
FAS-Picard($m_{\mathcal{A}}$) refers to the full approximation scheme with one Picard iteration and $m_{\mathcal{A}}$ local spectral pressure degrees of freedom per aggregate---note that $m_{\mathcal{A}}$ is uniform in the domain.
We denote the number of flux degrees of freedom per coarse interface by $m_f$, and, unless specified otherwise, we set $m_f = 1$.
Nonlinear solver convergence is achieved either when the relative residual drops below $10^{-8}$, or when the absolute residual drops below $10^{-10}$.
On the coarsest level, we approximate the solution by using up to 10 nonlinear smoothing steps, unless convergence is achieved before this threshold.
In Algorithm~\ref{alg:backtracking}, we set the maximum number of backtracking steps to $n_{\max} = 4$, and we use $\theta = 0.9$.

The discrete problems are generated using our own implementation based on MFEM \cite{mfem} of the TPFA finite-volume scheme described in Section~\ref{sec:model}.
The multilevel spectral coarsening is performed with smoothG \cite{smoothg}, and the visualization is generated with GLVis \cite{glvis}.
All the experiments are executed on an Intel Core i7-4980HQ processor (clocked at 2.8 GHz) with 16 GB of memory.

\subsection{Example 1: mildly heterogeneous permeability field based on the Egg model}

The mesh and permeability field used in this example are taken from the Egg model \cite{EggModel}.
The Cartesian mesh contains 18,553 active cells, and the (uniform) size of each cell is 8 m $\times$ 8 m $\times$ 4 m.
To obtain the reference permeability, $\mathbb{K}_0$, we consider realization 27 in the Egg model data set \cite{EggModelData}, which represents a mildly heterogeneous and channelized permeability field.
Then, we scale the permeability field by a constant such that the maximum permeability value in the $x$ and $y$ directions (respectively, in the $z$ direction) is equal to one (respectively, to $10^{-2}$); see Fig.~\ref{fig:egg}
For the boundary conditions, we impose $p = 0$  on $\partial\Omega \cap \{y = 0 \}$, and a no-flow condition on the remainder of the domain boundary.
The diagonal coefficient is 
\begin{equation}
\mathbb{K}(p) =  \mathbb{K}_0 e^{\alpha p},
\label{eq:permeability_pressure_relationship_numerical_example_1}  
\end{equation}
and the non-uniform source term is given by $f(y) = 0.025 \max\{ e^{\frac{y - 432}{480}}, 1 \}$.
\def\scaletwo{.45}
\begin{figure}[h!]
\centering
\includegraphics[scale=0.25,clip,trim=5cm 5cm 9cm 9cm]{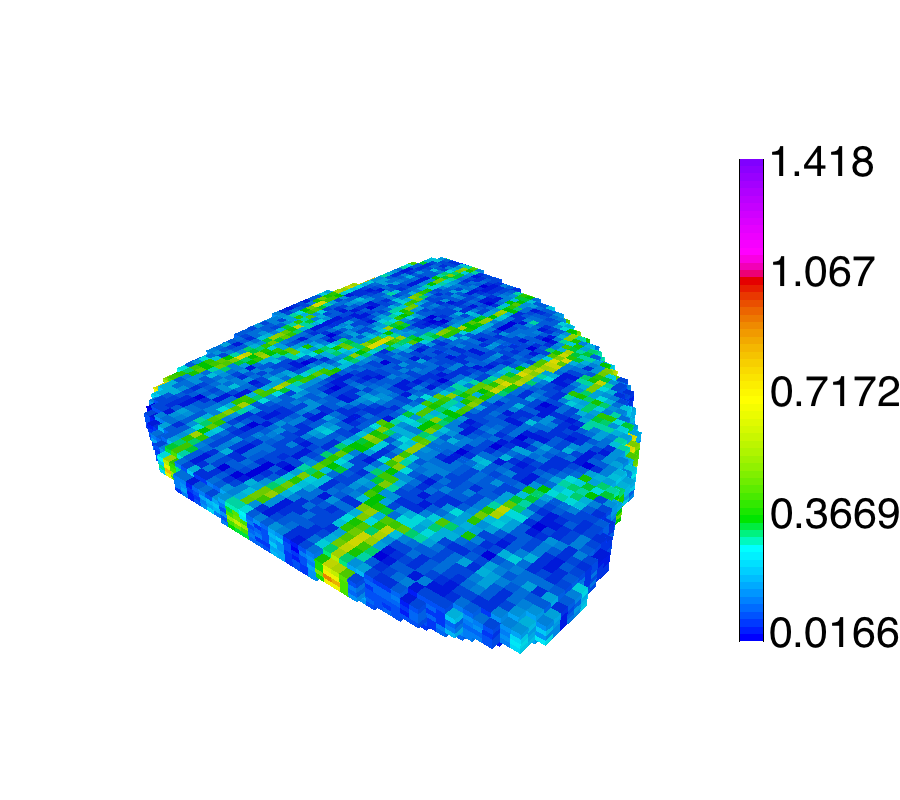}
\includegraphics[scale=0.2,clip,trim=23cm 5cm 0cm 4.6cm]{egg_scale.png}
\caption{Frobenius norm of the rescaled reference permeability, $\mathbb{K}_0$, in the Egg model (Example 1). The original channelized permeability field is obtained from realization 27 of \cite{EggModelData}.}
\label{fig:egg}
\end{figure}

The coarsening factor from the fine level to the first coarse level---i.e., from $\ell=0$ to $\ell=1$---is 28, while the coarsening factor from a coarse level ($\ell \geq 1$) to the next is 8.
We remark that the complexity of the coarse problem is heavily dependent on the stencil of the coarse flux degrees of freedom.
Since we have observed that, for this example, the coarse problem already represents a very good approximation to the original fine problem with one degree of freedom per coarse interface, we only consider the case $m_f = 1$ to minimize the complexity of the coarse problems.

In this example, the pressure-dependent multiplier appearing in \eqref{eq:permeability_pressure_relationship_numerical_example_1} is highly nonlinear.
As mentioned in Remark~\ref{rmk:special_backtracking}, we have implemented an additional backtracking algorithm designed to tackle the specific functional form of the permeability function.
More precisely, since the reference permeability coefficient, $\mathbb{K}_0$, is scaled by an exponential function of pressure, this additional backtracking aims at avoiding extremely large values in $\mathbb{K}$ that would undermine nonlinear convergence.
To achieve that, we limit the maximum change in pressure resulting from a smoothing step to $\log(1.5)/\alpha$.
That is, if $\max( \text{abs}( \Delta \Vec{p}_k ) )$---where the absolute value is applied component-wise to the entries of $\Delta \Vec{p}_k$---is larger than $\log(1.5)/\alpha$, then we scale $\Delta \Vec{x}_k$ by a constant so the new value of $\max( \text{abs}( \Delta \Vec{p}_k ) )$ is $\log(1.5)/\alpha$.
This specialized backtracking algorithm is applied at each nonlinear smoothing step, followed by the generic residual-based backtracking line-search of Algorithm~\ref{alg:backtracking}.
To perform a fair comparison, this two-step backtracking approach is used for the FAS solvers as well as for their single-level counterparts.

We are interested in comparing the performance of the nonlinear solvers as $\alpha$ is increased, since this parameter controls the nonlinearity of the problem.
To form a larger problem than the original setup described in \cite{EggModel}, we refine the original mesh in each direction.
On this refined mesh containing 148,428 active cells, we construct a four-level hierarchy.
We report the solution time and the nonlinear iteration counts in Table~\ref{tab:egg_alpha}.
Figure~\ref{fig:egg_alpha}(a) shows the solution times of single-level Picard and FAS-Picard, while Fig.~\ref{fig:egg_alpha}(b) focuses on single-level Newton and FAS-Newton.
In terms of nonlinear iteration count, Table~\ref{tab:egg_alpha} demonstrates that FAS-Picard and FAS-Newton remain very robust as the nonlinearity of the problem becomes more severe.
In terms of efficiency, Figs.~\ref{fig:egg_alpha}(a) and \ref{fig:egg_alpha}(b) clearly show that FAS-Picard and FAS-Newton achieve a better performance than the single-level solvers for all the considered values of $\alpha$.
We also note that the solution time reduction achieved with the FAS solvers becomes larger as $\alpha$ is increased---except for FAS-Picard($m_{\mathcal{A}}$ = 1); see Fig.~\ref{fig:egg_alpha}(a).

Considering now the impact of the enrichment of the coarse space on robustness and efficiency, we remark that increasing the number of spectral pressure degrees of freedom, $m_{\mathcal{A}}$, from 1 to 7, yields a reduction in the nonlinear iteration counts, and, for sufficiently large values of $\alpha$, a reduction in solution time.
This is particularly significant for FAS-Picard with $\alpha \geq 0.2$.
\begin{table}[h!]
  \centering
  \small
  \begingroup
  \setlength{\tabcolsep}{6pt}
  \begin{tabular}{l
                  S[ table-figures-integer = 2,
                     table-figures-decimal = 2] 
                  S[ table-number-alignment = center,
                     table-figures-integer = 2,
                     table-figures-decimal = 0]
                  c
                  S[ table-figures-integer = 2,
                     table-figures-decimal = 2] 
                  S[ table-number-alignment = center,
                     table-figures-integer = 2,
                     table-figures-decimal = 0]
                  c
                  S[ table-figures-integer = 2,
                     table-figures-decimal = 2] 
                  S[ table-number-alignment = center,
                     table-figures-integer = 2,
                     table-figures-decimal = 0]
                  c
                  S[ table-figures-integer = 2,
                     table-figures-decimal = 2] 
                  S[ table-number-alignment = center,
                     table-figures-integer = 2,
                     table-figures-decimal = 0]
                  c
                  S[ table-figures-integer = 2,
                     table-figures-decimal = 2] 
                  S[ table-number-alignment = center,
                     table-figures-integer = 2,
                     table-figures-decimal = 0]
                 }           
    \toprule
    Solver
    & \multicolumn{2}{c}{$\alpha = 0.1$} && \multicolumn{2}{c}{$\alpha = 0.2$}
    && \multicolumn{2}{c}{$\alpha = 0.4$} && \multicolumn{2}{c}{$\alpha = 0.8$}
    && \multicolumn{2}{c}{$\alpha = 1.6$} \\
    \cmidrule{2-3} \cmidrule{5-6} \cmidrule{8-9} \cmidrule{11-12} \cmidrule{14-15}
    &  \multicolumn{1}{c}{$T_{\text{sol}}$} & \multicolumn{1}{c}{$n_{\text{it}}$}
    && \multicolumn{1}{c}{$T_{\text{sol}}$} & \multicolumn{1}{c}{$n_{\text{it}}$}
    && \multicolumn{1}{c}{$T_{\text{sol}}$} & \multicolumn{1}{c}{$n_{\text{it}}$}
    && \multicolumn{1}{c}{$T_{\text{sol}}$} & \multicolumn{1}{c}{$n_{\text{it}}$}
    && \multicolumn{1}{c}{$T_{\text{sol}}$} & \multicolumn{1}{c}{$n_{\text{it}}$} \\
    \midrule
    Single-level Picard    & 9.55 & 14     && 12.74 & 17       && 15.46 & 20       && 20.13 & 26       && 26.19 & 33       \\
    FAS-Picard ($m_\A=1$)  & 8.63 &  5     &&  8.84 &  5       && 12.01 &  7$^{*}$ && 16.97 &  9       && 23.44 & 13$^{*}$ \\
    FAS-Picard ($m_\A=7$)  & 6.90 &  4$^*$ &&  8.34 &  4       && 10.65 &  5       && 13.05 &  6       && 15.77 &  7       \\
    \midrule
    Single-level Newton    & 6.86 &  6     && 10.03 &  8       && 12.01 &  9       && 15.00 & 11       && 18.42 & 13       \\
    FAS-Newton ($m_\A=1$)  & 4.05 &  2     &&  5.67 &  3$^{*}$ &&  5.89 &  3$^{*}$ &&  7.78 &  4$^{*}$ && 7.83  &  4$^{*}$ \\
    FAS-Newton ($m_\A=7$)  & 4.53 &  2     &&  4.56 &  2       &&  4.63 &  2       &&  6.47 &  3$^{*}$ && 7.11  &  3       \\
    \bottomrule
  \end{tabular}
  \endgroup
  \caption{solution time ($T_\text{sol.}$) in seconds and nonlinear iteration count ($n_{\text{it.}}$) for the refined problem based on the Egg model (Example 1). The superscript $^*$ next to a nonlinear iteration count is used to indicate that, during the last nonlinear iteration, the V-cycle was terminated because convergence was achieved after the nonlinear pre-smoothing step. The mesh has 148,424 cells that we aggregate to construct a three-level hierarchy in the FAS solvers.} 
\label{tab:egg_alpha}
\end{table}
\pgfplotstableread
{
colnames		0		1		2		3		4
alpha		0.1 		0.2		0.4		0.8		1.6 
newton		6.86258	10.0257	12.0135	15.0034	18.416
fas_newton1	4.04871	5.67348	5.89212	7.77863	7.83071
fas_newton7	4.52615	4.56289	4.63055	6.4668	7.11059
picard		9.55366	12.7443	15.4459	20.1323	26.1932
fas_picard1	8.6277	8.84392	12.0103	16.9715	23.4442
fas_picard7	6.90058	8.34036	10.6486	13.0522	15.774
}\EggAlphaData
\pgfplotstabletranspose[colnames from=colnames]\EggAlphaData{\EggAlphaData}

\begin{figure}[h!]
  \centering
  \subcaptionbox{Nonlinear smoothing: Picard iteration}
  {
  \begin{tikzpicture}
    \begin{semilogxaxis}[
                  width=.5\textwidth,
                  height=.35\textwidth,
                  grid = major,
                  major grid style={very thin,draw=gray!30},
                  xmin=0,xmax=2,
                  ymin=0,ymax=30,
                  xlabel={$\alpha$},
                  ylabel={solution time (s)},
                  xtick=data,
                  ytick={0,5,...,30},
                  ylabel near ticks,
                  xlabel near ticks,
                  legend style={font=\small},
                  tick label style={font=\small},
                  xticklabels from table={\EggAlphaData}{alpha},
                  label style={font=\small},
                  legend entries={Single-level Picard, FAS-Picard($m_\A=1$), FAS-Picard($m_\A=7$)},
                  legend cell align={left},
                  legend pos=north west,
                  ]
      \addplot [mark=square, skyblue1, thick] table [x=alpha,y=picard] {\EggAlphaData};
      \addplot [mark=o, scarletred1, thick] table [x=alpha, y=fas_picard1] {\EggAlphaData};
      \addplot [mark=asterisk, chocolate1, thick] table [x=alpha, y=fas_picard7] {\EggAlphaData};
    \end{semilogxaxis}
  \end{tikzpicture}
  }
  \hspace{3mm}
  \subcaptionbox{Nonlinear smoothing: Newton iteration}
  {
  \begin{tikzpicture}
    \begin{semilogxaxis}[
                  width=.5\textwidth,
                  height=.35\textwidth,
                  grid = major,
                  major grid style={very thin,draw=gray!30},
                  xmin=0,xmax=2,
                  ymin=0,ymax=30,
                  xlabel={$\alpha$},
                  ylabel={solution time (s)},
                  xtick=data,
                  ytick={0,5,...,30},
                  ylabel near ticks,
                  xlabel near ticks,
                  legend style={font=\small},
                  tick label style={font=\small},
                  xticklabels from table={\EggAlphaData}{alpha},
                  label style={font=\small},
                  legend entries={Single-level Newton, FAS-Newton($m_\A=1$), FAS-Newton($m_\A=7$)},
                  legend cell align={left},
                  legend pos=north west,
                  ]
      \addplot [mark=square, skyblue1, thick] table [x=alpha,y=newton] {\EggAlphaData};
      \addplot [mark=o, scarletred1, thick] table [x=alpha, y=fas_newton1] {\EggAlphaData};
      \addplot [mark=asterisk, chocolate1, thick] table [x=alpha, y=fas_newton7] {\EggAlphaData};
    \end{semilogxaxis}
  \end{tikzpicture}
  }
  \caption{solution time in seconds as a function of the parameter, $\alpha$, controlling the nonlinearity of the problem. The mesh has 148,424 cells and is obtained by refining the original Egg model (Example 1). We construct a three-level hierarchy in the FAS solvers.}
  \label{fig:egg_alpha}
\end{figure}

Next, we assess the scalability of the nonlinear solvers with respect to problem size.
To this end, we uniformly refine the original mesh of \cite{EggModel} to generate two larger meshes.
The number of cells as well as the number of FAS levels for each mesh are reported in Table~\ref{tab:egg_levels}.
To make the problem highly nonlinear, we set $\alpha = 1.6$.
The solution times are shown in Fig.~\ref{fig:egg_refine}.
All solvers appear to be scalable in the sense that their solution time is proportional to the problem size.
FAS solvers in this case are also faster than the single level counterparts.

\begin{table}[h!]
  \centering
  \small
  \begin{tabular}{llll}
    \toprule
    Number of cells  & 18,553 & 148,424  & 1,187,392\\
    \midrule
    Number of levels & 3 & 4 & 5 \\
    \bottomrule
  \end{tabular}
  \caption{Number of cells and FAS levels for the three problems based on the Egg model (Example 1) used to assess solver scalability in Fig.~\ref{fig:egg_refine}.}
  \label{tab:egg_levels}
\end{table}

\pgfplotstableread
{
colnames		0		1		2
cells			18553	176792	1187392
cell_labels			18,553	176,792	1,187,392
linear		0.4875	4.6425	31.2
newton		1.9434	18.416	169.052
fas_newton1	0.98701	7.8307	71.7389
fas_newton7	0.91777	7.1106	62.3366	
picard      		3.73687	26.1932 	226.518
fas_picard1 	2.19723 	23.4442	183.597
fas_picard7	2.0329	15.774	140.43
}\EggRefineData
\pgfplotstabletranspose[colnames from=colnames]\EggRefineData{\EggRefineData}

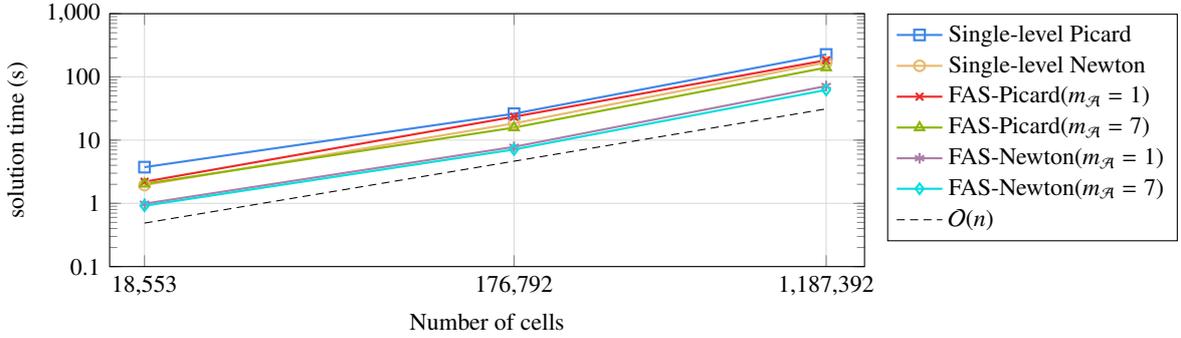
\begin{figure}[h!]
  \centering
  \begin{tikzpicture}
    \begin{loglogaxis}[
                  width=.7\textwidth,
                  height=.3\textwidth,
                  grid = major,
                  major grid style={very thin,draw=gray!30},
                  xmin=15000,xmax=1500000,
                  ymin=0.1,ymax=1000,
                  xlabel={Number of cells},
                  ylabel={solution time (s)},
                  xtick=data,
                  ytick distance=10^1,
                  ylabel near ticks,
                  xlabel near ticks,
                  legend style={font=\small},
                  tick label style={font=\small},
                  xticklabels from table={\EggRefineData}{cell_labels},
                  log ticks with fixed point,
                  label style={font=\small},
                  legend entries={Single-level Picard, Single-level Newton, FAS-Picard($m_{\mathcal{A}} = 1$), 
                  			  FAS-Picard($m_{\mathcal{A}} = 7$), FAS-Newton($m_{\mathcal{A}} = 1$), FAS-Newton($m_{\mathcal{A}} = 7$), $\mathcal{O}(n)$},
                  legend cell align={left},
                  legend pos=outer north east,
                  ]
      \addplot [mark=square, skyblue1, thick] table [x=cells, y=picard] {\EggRefineData};
      \addplot [mark=o, chocolate1, thick] table [x=cells, y=newton] {\EggRefineData};
      \addplot [mark=x, scarletred1, thick] table [x=cells, y=fas_picard1] {\EggRefineData};
      \addplot [mark=triangle, applegreen, thick] table [x=cells, y=fas_picard7] {\EggRefineData};
      \addplot [mark=asterisk, plum1, thick] table [x=cells, y=fas_newton1] {\EggRefineData};
      \addplot [mark=diamond, blue-green, thick] table [x=cells, y=fas_newton7] {\EggRefineData};
      \addplot [no marks, densely dashed] table [x=cells, y=linear] {\EggRefineData};
    \end{loglogaxis}
  \end{tikzpicture}
  \caption{solution time in seconds as a function of problem size for the three problems based on the Egg model (Example 1). We set $\alpha = 1.6$ to make the problem highly nonlinear. The number of levels in the FAS solvers is given in Table~\ref{tab:egg_levels}.}
  \label{fig:egg_refine}
\end{figure}

\subsection{Example 2: highly heterogeneous and anisotropic permeability field based on the SPE10 model}

We consider the Cartesian grid and permeability field from the SPE10 comparison solution project \cite{spe10}, which is characterized by its high heterogeneity and its anisotropy.
The domain, in meters, is $\Omega = (0, 365.76) \times (0, 670.56) \times (0, 51.816)$.
It is uniformly partitioned into 60 $\times$ 220 $\times$ 85 cells, so that the dimension of each cell is 6.096 m $\times$ 3.048 m $\times$ 0.6096 m.
The nonlinear permeability function $\mathbb{K}(p)$ is given in \eqref{eq:permeability_pressure_relationship_numerical_example_1}, and the original permeability field is rescaled using the same methodology as in Example 1; see Fig.~\ref{fig:spe10}.
The source term is given by $f(y) = 0.000025 \max\{ e^{\frac{y - 1980}{2200}}, 1 \}$.
We impose $p = 0$ on $\partial\Omega \cap \{y = 0 \}$, and a no-flow condition on the remainder of the boundary. 
\begin{figure}[t]
\centering
\includegraphics[scale=0.7]{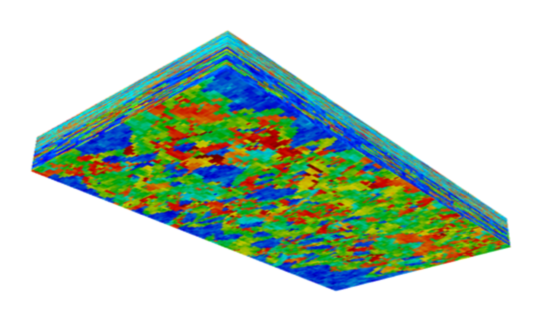}
\includegraphics[scale=0.2,clip,trim=25cm 4.6cm 0 4.6cm]{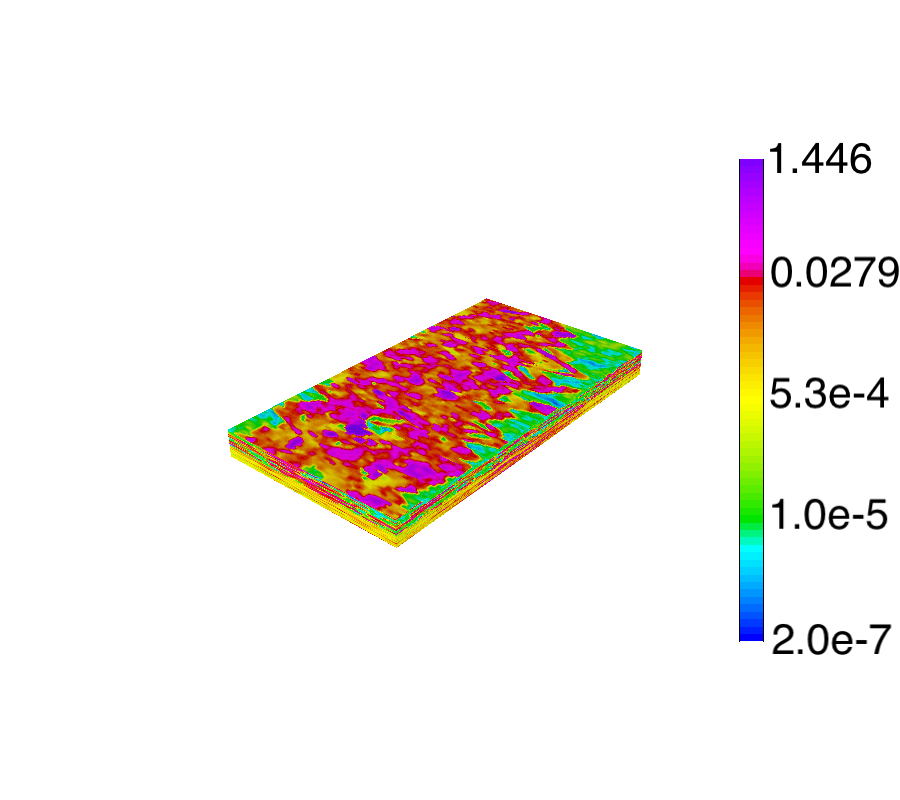}
\caption{Frobenius norm of the rescaled reference permeability, $\mathbb{K}_0$, in the SPE10 model (Example 2).}
\label{fig:spe10}
\end{figure}

The coarsening factor from the fine level to the first coarse level---i.e., from $\ell=0$ to $\ell=1$---is 256, while the coarsening factor from a coarse level ($\ell \geq 1$) to the next is 8. 
The FAS solvers rely on a three-level hierarchy of approximation spaces.
We use the same two-step backtracking line-search algorithm as in Example 1 with a maximum absolute change in pressure between two smoothing steps set to $\log(5)/\alpha$.

We note that compared the previous example, this is a much more difficult problem for the nonlinear solvers due to the higher heterogeneity of the permeability field.
For this reason, we only consider the Newton iteration for the smoothing steps of the single-level and multilevel solvers.
We have also observed that more flux degrees of freedom per interface are needed to improve the approximation property of the coarse space.
Since increasing the number of flux degrees of freedom per interface, $m_f$, results in a higher complexity of the coarse problem, we set the number of pressure degrees of freedom per aggregate to one ($m_{\mathcal{A}} = 1$).
We have found that, for this example, this yields better results than increasing $m_{\mathcal{A}}$ and setting $m_f = 1$ as in Example 1.

As in Example 1, we want to demonstrate the robustness and efficiency of the FAS solvers as nonlinearity becomes more severe.
This is done by gradually increasing the parameter, $\alpha$, controlling the nonlinearity of the permeability function \eqref{eq:permeability_pressure_relationship_numerical_example_1}.
In Table~\ref{tab:spe_alpha}, the number of nonlinear iterations as well as the solution time are reported.
We see that for single-level Newton, both the nonlinear iteration count and the solution time double when $\alpha$ is increased from 0.1 to 1.6.
Instead, the FAS-Newton solvers remain robust when nonlinearity becomes stronger.
For instance, when $\alpha$ goes from 0.1 to 1.6, the number of nonlinear iterations performed by FAS-Newton($m_f = 1$) only increases by one, and its solution time by about 54\%.
We also see that FAS-Newton($m_f = 3$) is even less sensitive to $\alpha$, as the instability that results in relatively large solution time increase for $\alpha = 0.4$ with FAS-Newton($m_f = 1$) vanishes with FAS-Newton($m_f = 3$).

\begin{table}[h!]
  \centering
  \small
  \begingroup
  \setlength{\tabcolsep}{6pt}
  \begin{tabular}{l
                  S[ table-figures-integer = 2,
                     table-figures-decimal = 2] 
                  S[ table-number-alignment = center,
                     table-figures-integer = 2,
                     table-figures-decimal = 0]
                  c
                  S[ table-figures-integer = 2,
                     table-figures-decimal = 2] 
                  S[ table-number-alignment = center,
                     table-figures-integer = 2,
                     table-figures-decimal = 0]
                  c
                  S[ table-figures-integer = 2,
                     table-figures-decimal = 2] 
                  S[ table-number-alignment = center,
                     table-figures-integer = 2,
                     table-figures-decimal = 0]
                  c
                  S[ table-figures-integer = 2,
                     table-figures-decimal = 2] 
                  S[ table-number-alignment = center,
                     table-figures-integer = 2,
                     table-figures-decimal = 0]
                  c
                  S[ table-figures-integer = 2,
                     table-figures-decimal = 2] 
                  S[ table-number-alignment = center,
                     table-figures-integer = 2,
                     table-figures-decimal = 0]
                 }           
    \toprule
    Solver
    & \multicolumn{2}{c}{$\alpha = 0.1$} && \multicolumn{2}{c}{$\alpha = 0.2$}
    && \multicolumn{2}{c}{$\alpha = 0.4$} && \multicolumn{2}{c}{$\alpha = 0.8$}
    && \multicolumn{2}{c}{$\alpha = 1.6$} \\
    \cmidrule{2-3} \cmidrule{5-6} \cmidrule{8-9} \cmidrule{11-12} \cmidrule{14-15}
    &  \multicolumn{1}{c}{$T_{\text{sol}}$} & \multicolumn{1}{c}{$n_{\text{it}}$}
    && \multicolumn{1}{c}{$T_{\text{sol}}$} & \multicolumn{1}{c}{$n_{\text{it}}$}
    && \multicolumn{1}{c}{$T_{\text{sol}}$} & \multicolumn{1}{c}{$n_{\text{it}}$}
    && \multicolumn{1}{c}{$T_{\text{sol}}$} & \multicolumn{1}{c}{$n_{\text{it}}$}
    && \multicolumn{1}{c}{$T_{\text{sol}}$} & \multicolumn{1}{c}{$n_{\text{it}}$} \\
    \midrule
    Single-level Newton    & 372.94 & 20       && 487.05 & 25       && 550.94 & 30       && 643.57 & 35       && 746.77 & 41 \\
    FAS-Newton ($m_\A=1$)  & 120.60 &  4$^{*}$ && 128.59 &  4       && 336.89 &  7$^{*}$ && 152.89 &  5$^{*}$ && 185.68 &  5 \\
    FAS-Newton ($m_\A=3$)  & 119.20 &  3       && 139.59 &  4$^{*}$ && 157.82 &  4       && 164.10 &  5$^{*}$ && 179.90 &  5 \\
    \bottomrule
  \end{tabular}
  \endgroup
  \caption{Solution time ($T_\text{sol.}$) in seconds and nonlinear iteration count ($n_{\text{it.}}$) for the SPE10 model (Example 2). The superscript $^*$ next to a nonlinear iteration count is used to indicate that, during the last nonlinear iteration, the V-cycle was terminated because convergence was achieved after the nonlinear pre-smoothing step. We aggregate the cells of the fine mesh to construct a three-level hierarchy in the FAS solvers.}
\label{tab:spe_alpha}
\end{table}

The high nonlinearity of this problem also led us to assess the robustness of a nonlinear cascadic multigrid based on spectral coarsening, as an alternative to the V-cycle of the full approximation scheme.
In the nonlinear cascadic algorithm, the nonlinear iteration starts at the coarsest level.
The coarsest problem is solved nonlinearly until convergence, and its solution is used as an initial guess to solve the finer-level problem.
This process is repeated until the finest problem ($\ell = 0$) is solved.
In Table~\ref{tab:spe10_FMG}, we report the number of Newton iterations on each level.
We observe that while the number of nonlinear iterations on the coarsest level increases as $\alpha$ increases, the nonlinear iteration counts on the other levels---and importantly, on the finest level---remain very stable.

\begin{table}
\centering
  \small
  \begingroup
  \setlength{\tabcolsep}{5pt}
  \begin{tabular}{l
                  S[ table-number-alignment = center,
                     table-figures-integer = 2,
                     table-figures-decimal = 0] 
                  S[ table-number-alignment = center,
                     table-figures-integer = 2,
                     table-figures-decimal = 0] 
                  S[ table-number-alignment = center,
                     table-figures-integer = 2,
                     table-figures-decimal = 0]
                  c
                  S[ table-number-alignment = center,
                     table-figures-integer = 2,
                     table-figures-decimal = 0] 
                  S[ table-number-alignment = center,
                     table-figures-integer = 2,
                     table-figures-decimal = 0] 
                  S[ table-number-alignment = center,
                     table-figures-integer = 2,
                     table-figures-decimal = 0] 
                  c
                  S[ table-number-alignment = center,
                     table-figures-integer = 2,
                     table-figures-decimal = 0] 
                  S[ table-number-alignment = center,
                     table-figures-integer = 2,
                     table-figures-decimal = 0] 
                  S[ table-number-alignment = center,
                     table-figures-integer = 2,
                     table-figures-decimal = 0] 
                  c
                  S[ table-number-alignment = center,
                     table-figures-integer = 2,
                     table-figures-decimal = 0] 
                  S[ table-number-alignment = center,
                     table-figures-integer = 2,
                     table-figures-decimal = 0] 
                  S[ table-number-alignment = center,
                     table-figures-integer = 2,
                     table-figures-decimal = 0] 
                  c
                  S[ table-number-alignment = center,
                     table-figures-integer = 2,
                     table-figures-decimal = 0] 
                  S[ table-number-alignment = center,
                     table-figures-integer = 2,
                     table-figures-decimal = 0] 
                  S[ table-number-alignment = center,
                     table-figures-integer = 2,
                     table-figures-decimal = 0] 
                 }           
    \toprule
    Solver
    &  \multicolumn{3}{c}{$\alpha = 0.1$} && \multicolumn{3}{c}{$\alpha = 0.2$}
    && \multicolumn{3}{c}{$\alpha = 0.4$} && \multicolumn{3}{c}{$\alpha = 0.8$}
    && \multicolumn{3}{c}{$\alpha = 1.6$} \\
    \cmidrule(){2-4} \cmidrule(){6-8} \cmidrule(){10-12} \cmidrule(){14-16} \cmidrule(){18-20}
    & \multicolumn{1}{c}{$n_{\text{it}}^2$} & \multicolumn{1}{c}{$n_{\text{it}}^1$} & \multicolumn{1}{c}{$n_{\text{it}}^0$} &
    & \multicolumn{1}{c}{$n_{\text{it}}^2$} & \multicolumn{1}{c}{$n_{\text{it}}^1$} & \multicolumn{1}{c}{$n_{\text{it}}^0$} &
    & \multicolumn{1}{c}{$n_{\text{it}}^2$} & \multicolumn{1}{c}{$n_{\text{it}}^1$} & \multicolumn{1}{c}{$n_{\text{it}}^0$} &
    & \multicolumn{1}{c}{$n_{\text{it}}^2$} & \multicolumn{1}{c}{$n_{\text{it}}^1$} & \multicolumn{1}{c}{$n_{\text{it}}^0$} &
    & \multicolumn{1}{c}{$n_{\text{it}}^2$} & \multicolumn{1}{c}{$n_{\text{it}}^1$} & \multicolumn{1}{c}{$n_{\text{it}}^0$} \\
    \midrule
    Single-level Newton          & {--} & {--} & 20 && {--} & {--} & 25 && {--} & {--} & 30 && {--} & {--} & 35 && {--} & {--} & 41 \\
    Cascadic multigrid ($m_f=1$) &  14  &   5  &  5 &&  16  &   6  &  6 &&  20  &   5  &  6 &&  23  &   5  &  6  &&  26  &   6 &   6 \\
    Cascadic multigrid ($m_f=3$) &  12  &   5  &  6 &&  15  &   6  &  6 &&  17  &   6  &  6 &&  21  &   5  &  7  &&  24  &   5 &   6 \\
    \bottomrule
  \end{tabular}
  \endgroup
\caption{Nonlinear iteration count ($n_{\text{it}}^{\ell}$) as a function of $\alpha$ and the level $\ell$ for the SPE10 problem using the nonlinear Cascadic multigrid. The nonlinear problem at each level is solved until convergence, starting from the coarsest level ($\ell = 2$).}
\label{tab:spe10_FMG}
\end{table}

Figure~\ref{fig:spe_alpha}, which shows the solution time of each solver as a function of $\alpha$, illustrates that this robustness makes nonlinear cascadic multigrid more efficient than FAS-Newton.
It should be noted, however, that both multilevel solvers only exhibit a very modest increase in solution time from $\alpha = 0.1$ to $\alpha = 1.6$,  in contrast with the rapid increase observed with single-level Newton.

\pgfplotstableread
{
colnames		0		1		2		3		4
alpha		0.1 		0.2		0.4		0.8		1.6 
newton 372.941 487.05 550.943 643.571 746.765
fas_newton1 120.599 128.59 336.889 152.817 185.676
fas_newton3 119.201 139.593 157.823 164.101 179.903
fmg_newton1 89.139 88.1251 84.1998 90.0546 107.92
fmg_newton3 83.8132 90.229 92.1218 87.2057 102.971
}\SPEAlphaData
\pgfplotstabletranspose[colnames from=colnames]\SPEAlphaData{\SPEAlphaData}

\begin{figure}[h!]
  \centering
  \begin{tikzpicture}
    \begin{semilogxaxis}[
                  width=.7\textwidth,
                  height=.3\textwidth,
                  grid = major,
                  major grid style={very thin,draw=gray!30},
                  xmin=0, xmax=2.2,
                  ymin=0,ymax=800,
                  xlabel={$\alpha$},
                  ylabel={solution time (s)},
                  xtick=data,
                  ytick={0, 100, 200, 300, 400, 500, 600, 700, 800},
                  ylabel near ticks,
                  xlabel near ticks,
                  legend style={font=\small},
                  tick label style={font=\small},
                  xticklabels from table={\SPEAlphaData}{alpha},
                  label style={font=\small},
                  legend entries={Single-level Newton, FAS-V-Cycle($m_f=1$), FAS-V-Cycle($m_f=3$), 
                  			  Cascadic-MG($m_f=1$), Cascadic-MG($m_f=3$)},
                  legend cell align={left},
                  legend pos=outer north east,
                  ]
      \addplot [mark=triangle, skyblue1, thick] table [x=alpha,y=newton] {\SPEAlphaData};
      \addplot [mark=square, scarletred1, thick] table [x=alpha, y=fas_newton1] {\SPEAlphaData};
      \addplot [mark=asterisk, chocolate1, thick] table [x=alpha, y=fas_newton3] {\SPEAlphaData};
      \addplot [mark=o, plum1, thick] table [x=alpha, y=fmg_newton1] {\SPEAlphaData};
      \addplot [mark=x, applegreen, thick] table [x=alpha, y=fmg_newton3] {\SPEAlphaData};
    \end{semilogxaxis}
  \end{tikzpicture}
   \caption{Solution time in seconds as a function of the parameter, $\alpha$, controlling the nonlinearity of the permeability function in the SPE10 model (Example 2). The multilevel solvers rely on a three-level hierarchy. We consider the impact of the number of flux degrees of freedom per interface, $m_f$, on the solution time of the FAS solvers.}
   \label{fig:spe_alpha}
\end{figure}
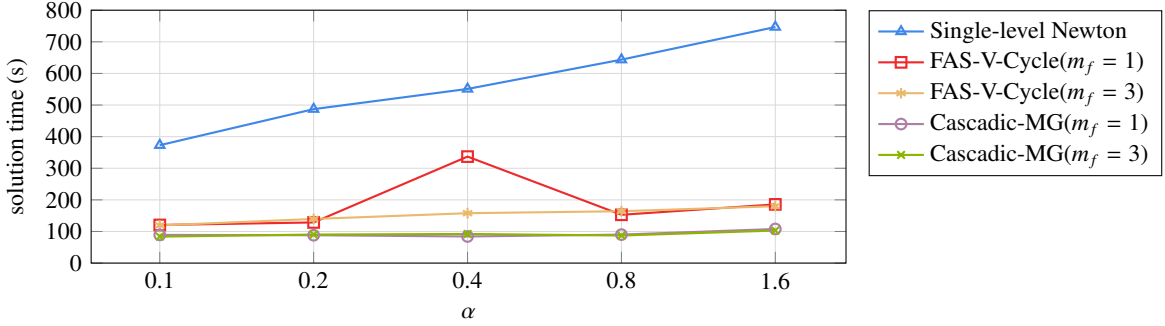

\subsection{Example 3: static Richards' equation}\label{sec:richard}

\begin{figure}
  \hfill
  \begin{subfigure}[b]{.45\linewidth}
    \centering
    \includegraphics[width=\linewidth]{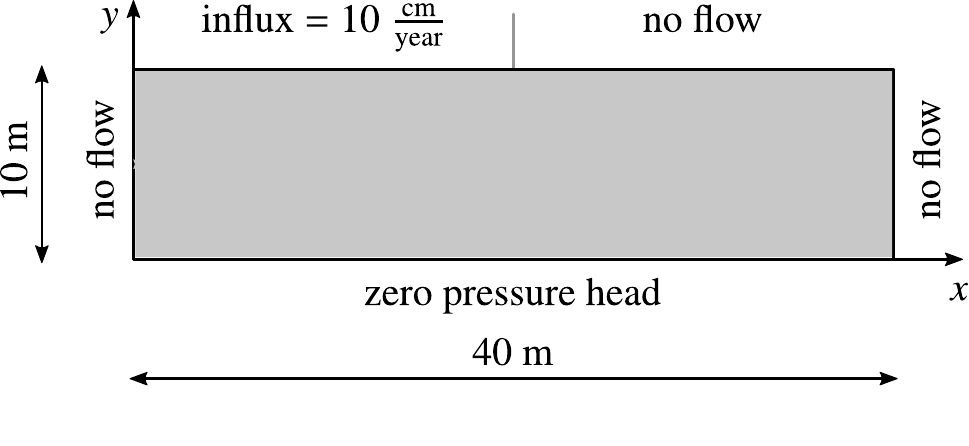}
    \caption{}
    \label{fig:richard_setup}
  \end{subfigure}
  \hfill
  \begin{subfigure}[b]{.45\linewidth}
    \centering
    \includegraphics[width=\linewidth]{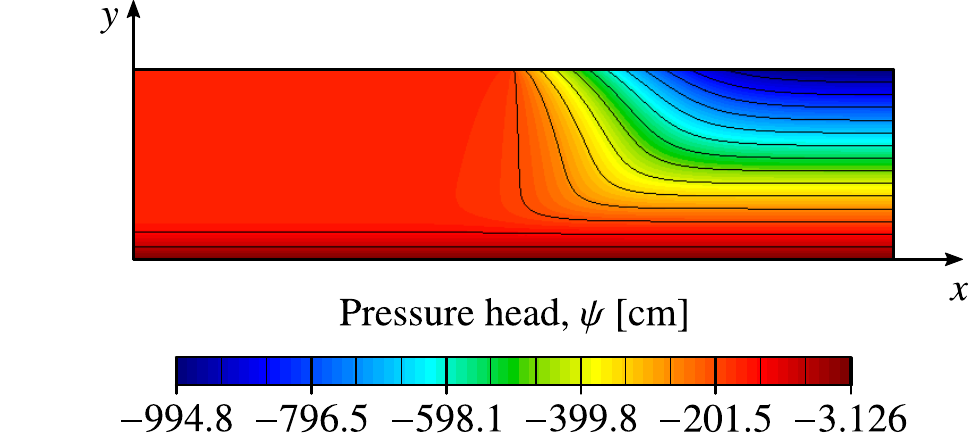}
    \caption{}
    \label{fig:richard_sol}
  \end{subfigure}
  \hfill
  \caption{Domain sketch (a) and pressure head solution (b) for the static Richards' equation (Example 3). It is obtained on the coarser mesh (160 $\times$ 40 cells) for the parameters corresponding to the porous sand medium given in Table~\ref{tab:richard-parameters}.}
\end{figure}

In the last example, we test our nonlinear multigrid solver on a nonlinear problem based on the static Richards' equation.
To do so, we replace \eqref{eq:Darcy_law}  with
\begin{equation}
\tensorOne{q} (\tensorOne{x}, \psi )
	=
  -\widehat{\mathbb{K}}(\tensorOne{x}, \psi ) \cdot \nabla (\psi + z) (\tensorOne{x}), \qquad
	 \tensorOne{x} \in \Omega,
\end{equation}
where $\psi$ denotes the pressure head, and $z$ is the height.
Note that by setting $p := \psi + z$ and $\mathbb{K}(p) := \widehat{\mathbb{K}}(p - z) = \widehat{\mathbb{K}}(\psi)$, we recover our model problem.
Note that here the diffusion tensor represents the hydraulic conductivity, having units of length per time.
We consider a test problem from \cite{mckeon87}.
The domain, in meters, is $\Omega = (0, 4000) \times(0, 1000)$.
We set the source term to zero ($f \equiv 0$).
We impose the boundary conditions depicted in Fig.~\ref{fig:richard_setup}.
The problem is highly nonlinear because of the constitutive relationship selected for $\widehat{\mathbb{K}}(\psi)$ \cite{mckeon87}:

\begin{equation}
  \widehat{\mathbb{K}}(\psi) = \mathbb{K}_0\frac{\alpha}{\alpha + | \psi |^\beta} = (k_0\mathbb{I})\frac{\alpha}{\alpha + | \psi |^\beta},
  \label{eq:permeability_function_richards}
\end{equation}
where $\mathbb{I}\in\mathbb{R}^{2\times 2}$ is the identity tensor in $\mathbb{R}^2$ and the values of $\alpha, \beta, k_0$ for different media are listed in Table~\ref{tab:richard-parameters}.
The pressure head solution is shown in Fig.~\ref{fig:richard_sol} for the parameters corresponding to the porous sand medium.

\begin{table}[h!]
\centering
\small
\begin{tabular}{l
                S[ table-number-alignment = left,
                   table-figures-integer = 1,
                   table-figures-decimal = 3,
                   table-figures-exponent = 1]
                S[ table-number-alignment = left,
                   table-figures-integer = 1,
                   table-figures-decimal = 2]
                S[ table-number-alignment = left,
                   table-figures-integer = 1,
                   table-figures-decimal = 3,
                   table-figures-exponent = 1]}
\toprule
 {Medium}        &
 \multicolumn{1}{l}{$\alpha$} &
 \multicolumn{1}{l}{$\beta$} &
 \multicolumn{1}{l}{$k_0$ [cm/day]} \\
 \midrule
 Loam            &   1.246e2    & 1.77 &   1.067e0 \\
 Sand            &   1.175e6    & 4.74 &   8.160e2 \\
 \bottomrule
\end{tabular}
\caption{Parameters used in the hydraulic conductivity function \eqref{eq:permeability_function_richards} for the loam and sand media. These parameters are the same as those used in \cite{mckeon87}.}
\label{tab:richard-parameters}
\end{table}

The initial mesh is uniformly discretized with 160 $\times$ 40 cells, and then it is regularly refined to create larger problems that we use to assess the scalability of the nonlinear solvers.
We set the coarsening factor to 36 for all levels.
For the problems run on the initial mesh, the FAS solvers are based on a three-level hierarchy.
When the refined meshes are used, the FAS solvers rely on four levels.
In this section, we only consider the residual-based backtracking line-search of Algorithm~\ref{alg:backtracking}.
For the loam medium, we set the maximum number of backtracking steps to $n^\text{loam}_{\max} = 4$ as in Examples 1 and 2.
However, since the parameters of Table~\ref{tab:richard-parameters} corresponding to sand medium result in a more nonlinear problem, we set $n^\text{sand}_{\max} = 10$.

 We remark that since the meshes are refined recursively, geometric multigrid could also be applied to solve the resulting nonlinear problems.
Indeed, that is the subject of study in \cite{mckeon87}.
Here, the purpose of refinement is solely to test the solver convergence with respect to problem size.
We treat each refined problem as an individual problem, and the coarsening does not use any refinement information.

The solution time and nonlinear iteration counts for the single-level and FAS nonlinear solvers are reported in Table~\ref{tab:richard-loam-backtrack} for the loam medium and Table~\ref{tab:richard-sand-backtrack} for the sand medium.
For the loam medium, we see that both FAS-Picard and FAS-Newton require fewer nonlinear iterations and can solve the four problems significantly faster than their single-level counterparts.
Figure~\ref{fig:richard_ref}, which shows the solution time of each solver as a function of problem size, illustrates that all solvers exhibit good scalability, and that FAS-Newton($m_{\mathcal{A}} = 13$) is the fastest solver for all problem sizes.

As mentioned above, the nonlinearity is stronger for the sand medium.
We have observed that all the Picard-based solvers fail to converge in this configuration.
However, for the Newton-based solvers, the FAS approach remains superior to the single-level approach in terms of both robustness---as shown by the smaller nonlinear iteration counts---and efficiency---as demonstrated by the significant reduction in solution time.
FAS-Newton($m_{\mathcal{A}} = 13$) still achieves the best performance among the solvers considered here.

We stress the fact that this example confirms that enriching the coarse space with more pressure degrees of freedom per aggregate improves the nonlinear behavior and reduces the solution times of the FAS solvers; see Tables~\ref{tab:richard-loam-backtrack} and \ref{tab:richard-sand-backtrack}.


\begin{table}[h!]
  \centering
  \small
  \begingroup
  \setlength{\tabcolsep}{6pt}
  \begin{tabular}{l
                  S[ table-figures-integer = 3,
                     table-figures-decimal = 2] 
                  S[ table-number-alignment = center,
                     table-figures-integer = 3,
                     table-figures-decimal = 0]
                  c
                  S[ table-figures-integer = 3,
                     table-figures-decimal = 2] 
                  S[ table-number-alignment = center,
                     table-figures-integer = 3,
                     table-figures-decimal = 0]
                  c
                  S[ table-figures-integer = 3,
                     table-figures-decimal = 2] 
                  S[ table-number-alignment = center,
                     table-figures-integer = 3,
                     table-figures-decimal = 0]
                  c
                  S[ table-figures-integer = 3,
                     table-figures-decimal = 2] 
                  S[ table-number-alignment = center,
                     table-figures-integer = 3,
                     table-figures-decimal = 0]
                 }            
    \toprule
    Solver
    &  \multicolumn{2}{c}{\#cells =  25,600} && \multicolumn{2}{c}{\#cells = 102,400}
    && \multicolumn{2}{c}{\#cells = 409,600} && \multicolumn{2}{c}{\#cells = 1,638,400} \\
    \cmidrule{2-3} \cmidrule{5-6} \cmidrule{8-9} \cmidrule{11-12}
    &  \multicolumn{1}{c}{$T_{\text{sol}}$} & \multicolumn{1}{c}{$n_{\text{it}}$}
    && \multicolumn{1}{c}{$T_{\text{sol}}$} & \multicolumn{1}{c}{$n_{\text{it}}$}
    && \multicolumn{1}{c}{$T_{\text{sol}}$} & \multicolumn{1}{c}{$n_{\text{it}}$}
    && \multicolumn{1}{c}{$T_{\text{sol}}$} & \multicolumn{1}{c}{$n_{\text{it}}$} \\
    \midrule
    Single-level Picard    & 8.32 & 121       && 35.10 & 123       && 162.00 & 126 && 734.40 & 127 \\
    FAS-Picard ($m_\A=1$)  & 5.36 &  33       && 16.70 &  27       &&  85.70 &  30 && 364.90 &  28 \\
    FAS-Picard ($m_\A=13$) & 4.06 &  24       && 11.70 &  18       &&  55.80 &  18 && 225.90 &  17$^{*}$ \\
    \midrule
    Single-level Newton    & 1.33 &  10       &&  5.39 &  10       &&  27.80 &  10 && 126.00 & 10 \\
    FAS-Newton ($m_\A=1$)  & 0.98 &   4$^{*}$ &&  3.99 &   4$^{*}$ &&  16.20 &   3 &&  76.70 &  3 \\
    FAS-Newton ($m_\A=13$) & 0.95 &   3       &&  2.82 &   2       &&  13.30 &   2 &&  56.30 &  2 \\
    \bottomrule
  \end{tabular}
  \endgroup
  \caption{Solution time ($T_\text{sol.}$) in seconds and nonlinear iteration count ($n_{\text{it.}}$) for the Richards' problem (Example 3) with the loam medium (Table~\ref{tab:richard-parameters}).  The superscript $^*$ next to a nonlinear iteration count is used to indicate that, during the last nonlinear iteration, the V-cycle was terminated because convergence was achieved after the nonlinear pre-smoothing step. We aggregate the cells of the fine mesh to construct a three-level hierarchy in the FAS solvers.} 
\label{tab:richard-loam-backtrack}
\end{table}

\begin{table}[h!]
\centering
  \small
  \begingroup
  \setlength{\tabcolsep}{6pt}
  \begin{tabular}{l
                  S[ table-figures-integer = 3,
                     table-figures-decimal = 2] 
                  S[ table-number-alignment = center,
                     table-figures-integer = 3,
                     table-figures-decimal = 0]
                  c
                  S[ table-figures-integer = 3,
                     table-figures-decimal = 2] 
                  S[ table-number-alignment = center,
                     table-figures-integer = 3,
                     table-figures-decimal = 0]
                  c
                  S[ table-figures-integer = 3,
                     table-figures-decimal = 2] 
                  S[ table-number-alignment = center,
                     table-figures-integer = 3,
                     table-figures-decimal = 0]
                  c
                  S[ table-figures-integer = 3,
                     table-figures-decimal = 2] 
                  S[ table-number-alignment = center,
                     table-figures-integer = 3,
                     table-figures-decimal = 0]
                 }            
    \toprule
    Solver
    &  \multicolumn{2}{c}{\#cells =  25,600} && \multicolumn{2}{c}{\#cells = 102,400}
    && \multicolumn{2}{c}{\#cells = 409,600} && \multicolumn{2}{c}{\#cells = 1,638,400} \\
    \cmidrule{2-3} \cmidrule{5-6} \cmidrule{8-9} \cmidrule{11-12}
    &  \multicolumn{1}{c}{$T_{\text{sol}}$} & \multicolumn{1}{c}{$n_{\text{it}}$}
    && \multicolumn{1}{c}{$T_{\text{sol}}$} & \multicolumn{1}{c}{$n_{\text{it}}$}
    && \multicolumn{1}{c}{$T_{\text{sol}}$} & \multicolumn{1}{c}{$n_{\text{it}}$}
    && \multicolumn{1}{c}{$T_{\text{sol}}$} & \multicolumn{1}{c}{$n_{\text{it}}$} \\
    \midrule
    Single-level Newton    & 2.33 & 17 && 13.00 & 22       && 61.40 & 23       && 268.80 & 22       \\
    FAS-Newton ($m_\A=1$)  & 2.09 &  7 &&  9.05 &  7$^{*}$ && 36.70 &  6       && 156.20 &  6$^{*}$ \\
    FAS-Newton ($m_\A=13$) & 1.90 &  6 &&  4.59 &  4$^{*}$ && 22.00 &  4$^{*}$ && 107.00 &  5$^{*}$ \\
    \bottomrule
  \end{tabular}
  \endgroup
  \caption{Solution time ($T_\text{sol.}$) in seconds and nonlinear iteration count ($n_{\text{it.}}$) for the Richards' problem (Example 3) with the sand medium (Table~\ref{tab:richard-parameters}). The Picard-based solvers do not converge for this highly nonlinear problem. The superscript $^*$ next to a nonlinear iteration count is used to indicate that, during the last nonlinear iteration, the V-cycle was terminated because convergence was achieved after the nonlinear pre-smoothing step. We aggregate the cells of the fine mesh to construct a three-level hierarchy in the FAS solvers.} 
  \label{tab:richard-sand-backtrack}
\end{table}

\pgfplotstableread
{
colnames		0		1		2		3
cells 25600 102400 409600 1638400
cell_labels 25,600 102,400 409,600 1,638,400
picard 8.32 35.1 162. 743.4
newton 1.33 5.39 27.8 126.
fas_picard1 5.36 16.7 85.7 364.9
fas_picard13 4.06 11.7 55.8 225.9
fas_newton1 0.98 3.99 16.2 76.7
fas_newton13 0.95 2.82 13.3 56.3
linear 0.6000    2.4000    9.6000   38.4000
}\RichardRefineData
\pgfplotstabletranspose[colnames from=colnames]\RichardRefineData{\RichardRefineData}

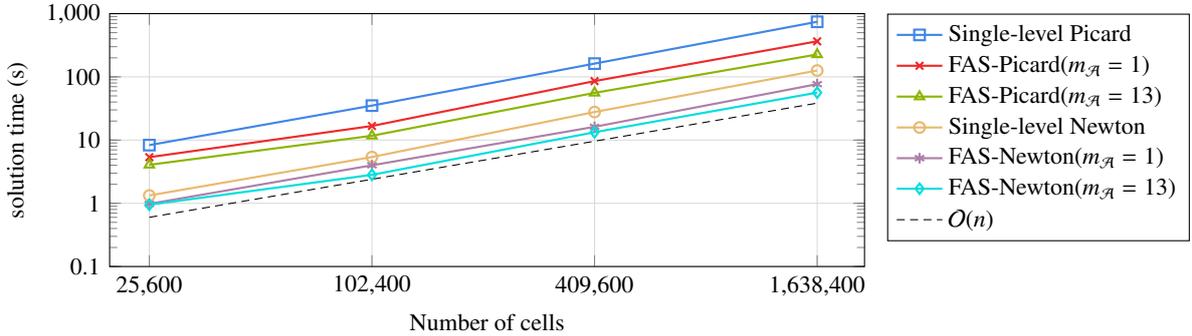
\begin{figure}[h!]
  \centering
  \begin{tikzpicture}
    \begin{loglogaxis}[
                  width=.7\textwidth,
                  height=.3\textwidth,
                  grid = major,
                  major grid style={very thin,draw=gray!30},
                  xmin=20000,xmax=2200000,
                  ymin=0.1,ymax=1000,
                  xlabel={Number of cells},
                  ylabel={solution time (s)},
                  ytick distance=10^1,
                  xtick=data,
                  ylabel near ticks,
                  xlabel near ticks,
                  legend style={font=\small},
                  tick label style={font=\small},
                  xticklabels from table={\RichardRefineData}{cell_labels},
                  log ticks with fixed point,
                  label style={font=\small},
                  legend entries={Single-level Picard, FAS-Picard($m_{\mathcal{A}} = 1$), FAS-Picard($m_{\mathcal{A}} = 13$), 
                  			  Single-level Newton, FAS-Newton($m_{\mathcal{A}} = 1$), FAS-Newton($m_{\mathcal{A}} = 13$), $\mathcal{O}(n)$},
                  legend cell align={left},
                  legend pos=outer north east,
                  ]
      \addplot [mark=square, skyblue1, thick] table [x=cells, y=picard] {\RichardRefineData};
      \addplot [mark=x, scarletred1, thick] table [x=cells, y=fas_picard1] {\RichardRefineData};
      \addplot [mark=triangle, applegreen, thick] table [x=cells, y=fas_picard13] {\RichardRefineData};
      \addplot [mark=o, chocolate1, thick] table [x=cells, y=newton] {\RichardRefineData};
      \addplot [mark=asterisk, plum1, thick] table [x=cells, y=fas_newton1] {\RichardRefineData};
      \addplot [mark=diamond, blue-green, thick] table [x=cells, y=fas_newton13] {\RichardRefineData};
      \addplot [no marks, densely dashed] table [x=cells, y=linear] {\RichardRefineData};
    \end{loglogaxis}
  \end{tikzpicture}  
  \caption{Solution time in seconds as a function of problem size for the four problems based on the Richards' equation (Example 3). We use the parameters of the loam medium given in Table~\ref{tab:richard-parameters}.}
  \label{fig:richard_ref}
\end{figure}

\section{Concluding remarks}\label{sec:conclude}

In this paper, we introduce a nonlinear multigrid solver for heterogeneous nonlinear diffusion problems based on the full approximation scheme.  It relies on
a TPFA finite volume scheme in mixed form and local spectral coarsening following the framework of \cite{barker17,ml-spectral-coarsening}. 
Nonlinear iteration is performed by solving a sequence of nonlinear problems on nested approximation spaces (levels).
The transfer operators between levels are obtained by solving local eigenvalue problems defined on cell aggregates.
The coarsening framework is applicable to problems discretized on general unstructured grids.
We compare the nonlinear multigrid solver with single-level Picard (fixed-point) and Newton iterations for several challenging examples numerical examples.
Results indicate that the nonlinear multigrid solver outperforms the single-level counterparts in two ways.
First, it exhibits a more robust convergence behavior characterized by smaller nonlinear iteration counts.
Second, it is algorithmically scalable and yields significant solution time reductions compared to the single-level approaches.
We have also shown that, in most cases, enriching the coarse space with more local eigenvectors leads to a reduction in solution time.
Another key feature of our coarsening strategy is that it is performed in a mixed (finite-volume) setting, which allows us to approximate fluxes directly.
Coarse fluxes are therefore, by construction, conservative on the coarse levels.
These properties are very attractive for the construction of multilevel nonlinear solvers for multiphase flow and transport simulations, the subject of our ongoing research.

\section*{Acknowledgements}
\label{sec::acknow}
Funding was provided by TOTAL S.A. through the FC-MAELSTROM project.
Portions of this work were performed under the auspices of the U.S. Department of
Energy by Lawrence Livermore National Laboratory under Contract DE-AC52-07-NA27344 (LLNL-JRNL-808398).

\appendix

\section{Construction of the vertex-based and edge-based approximation spaces} \label{sec:construction_of_prolongation_and_projection}

\subsection{Construction of the vertex-based pressure space}\label{sec:vertex-based-space}

Following the spectral coarsening methodology of \cite{ml-spectral-coarsening}, we construct the coarse pressure space aggregate by aggregate by solving local eigenvalue problems.
We use the graph-based terminology of  \cite{ml-spectral-coarsening}, in which graph vertices and edges are the dual representations of mesh cells and interfaces, respectively.
%
%
%
Let us consider a graph $G^{\ell} = (V^{\ell},E^{\ell})$ at level $\ell$, where $V^{\ell} = \{ v^{\ell}_K \}$ is the set of vertices and $E^{\ell} = \{ e^{\ell}_{K,L} \}$ is the set of edges.
To construct the coarse graph, $G^{\ell+1} = (V^{\ell+1},E^{\ell+1})$, we use METIS to partition $G^{\ell}$ into aggregates of fine vertices, $\{ \mathcal{A}^{\ell}_K \}$.
Then, in the coarse graph, each vertex, $v^{\ell+1}_K \in V^{\ell+1}$, represents one of these aggregates, $\mathcal{A}^{\ell}_K$.
We introduce a coarse edge, $e^{\ell+1}_{K,L} \in E^{\ell+1}$, if there exists a fine edge between two vertices belonging to $\mathcal{A}^{\ell}_K$ and $\mathcal{A}^{\ell}_L$, respectively. 
\begin{figure}[t]
  \hfill
  \begin{subfigure}[c]{.35\linewidth}
    \centering
    \includegraphics[width=\linewidth]{aggregates_graph}
    \caption{}
  \end{subfigure}
  \hfill
  \begin{subfigure}[c]{.35\linewidth}
    \centering
    \includegraphics[width=\linewidth]{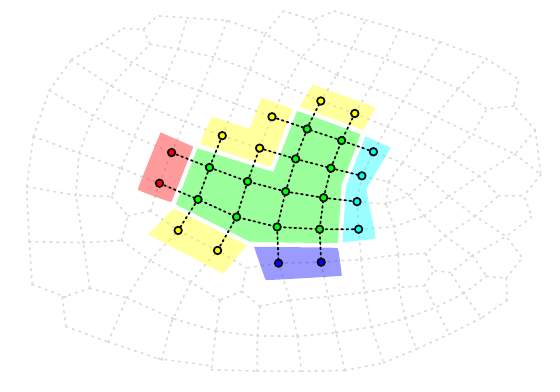}
    \caption{}
  \end{subfigure}
  \hfill\null
  \caption{\label{fig:extended-aggregate}
For the partitioned unstructured mesh shown in (a), consider the coarse aggregate $\mathcal{A}$ shown in green and located in the center.
The local spectral problem is solved on an extended aggregate, $\mathcal{N}(\mathcal{A})$ that consists of all the vertices in the colored cells of (b).
This allows for a well-defined treatment of the edge-based degrees of freedom when we use more than one per interface.}
\end{figure}

Let us consider now an aggregate of fine vertices, $\mathcal{A}^{\ell}_K$.
Let $\mathcal{N}( \mathcal{A}^{\ell}_K )$ be the set of vertices in $V^{\ell}$ that are connected to a fine vertex in $\mathcal{A}^{\ell}_K$ as illustrated in Fig.~\ref{fig:extended-aggregate}.
Denoting by $\Mat{M}^{\ell} |_{\mathcal{N}(\mathcal{A}^{\ell}_K)}$ the submatrix of $\Mat{M}^{\ell}$ restricted to the edges connecting the vertices  in $\mathcal{N}(\mathcal{A}^{\ell}_K)$, and by $\Mat{D}^{\ell} |_{\mathcal{N}(\mathcal{A}^{\ell}_K)}$ the submatrix of $\Mat{D}^{\ell}$ whose rows and columns are restricted to, respectively, the vertices in $\mathcal{N}(\mathcal{A}^{\ell}_K)$ and the edges connecting the vertices in $\mathcal{N}(\mathcal{A}^{\ell}_K)$, the local eigenvalue problem associated with aggregate $\mathcal{A}^{\ell}_K$ is
\begin{equation}
\Mat{D}^{\ell} |_{\mathcal{N}(\mathcal{A}^{\ell}_K)} (\Mat{M}^{\ell} |_{\mathcal{N}(\mathcal{A}^{\ell}_K)})^{-1} (D^{\ell} |_{\mathcal{N}(\mathcal{A}^{\ell}_K)})^T \Vec{q} = \lambda \Vec{q}.  
\label{eq:local-eigenvalue-problem}  
\end{equation}
We select the $m_{\mathcal{A}^{\ell}_K}$ eigenvectors corresponding to the $m_{\mathcal{A}^{\ell}_K}$ smallest eigenvalues in \eqref{eq:local-eigenvalue-problem}.
To remove any linear dependence between these eigenvectors, we apply a singular value decomposition (SVD) to the selected subset.
The remaining eigenvectors are used to form the columns of a rectangular, orthogonal matrix $\Mat{P}_{\mathcal{A}^{\ell}_K} := [\Vec{q}_1 \, \dots \, \Vec{q}_{m_{\mathcal{A}^{\ell}_K}}]$.
%
The matrices $\{ \Mat{P}_{\mathcal{A}^{\ell}_K} \}$ are then collected to construct the global block-diagonal pressure interpolation operator, $(\Mat{P}_p)^{\ell}_{\ell+1}$---whose (rectangular) diagonal block corresponding to aggregate $\mathcal{A}^{\ell}_K$ is $\Mat{P}_{\mathcal{A}^{\ell}_K}$.
The global projection operator for the pressure degrees of freedom is defined as
\begin{equation}
(\Mat{\Pi}_p)^{\ell} := (\Mat{P}_p)^{\ell}_{\ell+1} (\Mat{Q}_p)^{\ell+1}_{\ell},
\label{eq:projection_pi_p}    
\end{equation}
where $(\Mat{Q}_p)^{\ell+1}_{\ell} := \big( (\Mat{P}_p)^{\ell}_{\ell+1} \big)^T$.

In preparation for the construction of the edge-based flux space, reviewed in the next section, we note that it is shown in \cite{ml-spectral-coarsening} that for each aggregate $\mathcal{A}^{\ell}_K$, the rectangular, orthogonal matrix $\Mat{P}_{\mathcal{A}^{\ell}_K}$ can be written in the form
\begin{equation}
\Mat{P}_{\mathcal{A}^{\ell}_K} =
\big[
\Vec{q}^{PV}_{\mathcal{A}^{\ell}_K}
\, \, \,
\Mat{P}^{NPV}_{\mathcal{A}^{\ell}_K} 
\big],
\label{eq:structure_pressure_interpolation}
\end{equation}
where ``PV'' refers to the work of Pasciak and Vassilevski \cite{pasciak2008exact}.
Let $\Vec{1}^{0}$ be the vector of ones whose size is equal to the number of vertices on level $\ell = 0$.
Introducing the vector defined recursively by $\Vec{1}^{\ell+1} := (\Mat{Q}_p)^{\ell+1}_{\ell} \Vec{1}^{\ell}$, the first column of $\Mat{P}_{\mathcal{A}^{\ell}_K}$ in \eqref{eq:structure_pressure_interpolation} is defined as
\begin{equation}
\Vec{q}^{PV}_{\mathcal{A}^{\ell}_K}
=
\frac{1}{|| \Vec{1}^{\ell}|_{\mathcal{A}^{\ell}_K} ||}
\Vec{1}^{\ell}|_{\mathcal{A}^{\ell}_K}.
\label{eq:definition_q_pv}
\end{equation}

\subsection{Construction of the edge-based flux space}\label{sec:edge-based-space}

The coarse edge-based space consists of two types of degrees of freedom.
In this section, we first review the construction of the trace extensions at the coarse interfaces, and then we discuss the definition of the bubbles in the cell aggregates.

Each trace extension is associated with a coarse interface, $\mathcal{F}^{\ell}_{K,L}$, between two aggregates, $\mathcal{A}^{\ell}_K$ and $\mathcal{A}^{\ell}_L$, defined as
\begin{equation}
\mathcal{F}^{\ell}_{K,L} := \{ e^{\ell}_{K,L} = ( v^{\ell}_K, v^{\ell}_L ) \in E^{\ell} : v^{\ell}_K \in \mathcal{A}^{\ell}_K \, \text{and} \, v^{\ell}_L \in \mathcal{A}^{\ell}_L \}.
\end{equation}
%
%
At this interface, we define the trace coming from $\mathcal{A}^{\ell}_K$ and corresponding to the $i$th column of $\Mat{P}_{\mathcal{A}^{\ell}_K}$ as
\begin{equation}
\big(  (\Mat{M}^{\ell} |_{\mathcal{N}(\mathcal{A}^{\ell}_K)})^{-1} (D^{\ell} |_{\mathcal{N}(\mathcal{A}^{\ell}_K)})^T \Vec{q}_i \big) \big|_{\Sigma^{\ell}(\mathcal{F}^{\ell}_{K,L})},  
\label{eq:definition-traces}  
\end{equation}
where $\Sigma^{\ell}(\mathcal{F}^{\ell}_{K,L})$ is the span of the edge-based degrees of freedom associated with interface $\mathcal{F}^{\ell}_{K,L}$.
The traces coming from $\mathcal{A}^{\ell}_L$ are defined in a similar fashion.
Following \citep{ml-spectral-coarsening}, we also introduce the vector $\Vec{\sigma}^{PV}_{\mathcal{F}^{\ell}_{K,L}}$  obtained by first solving the following local problem on the union of the two aggregates, $\mathcal{A}^{\ell}_{\mathcal{F}^{\ell}_{K,L}} := \mathcal{A}^{\ell}_{K} \cup \mathcal{A}^{\ell}_{L}$:
\begin{equation}
 \begin{bmatrix}
     \Mat{M}^{\ell} |_{\mathcal{A}_{\mathcal{F}^{\ell}_{K,L}}} & (\Mat{D}^{\ell} |_{\mathcal{A}_{\mathcal{F}^{\ell}_{K,L}}})^{T} \\
     \Mat{D}^{\ell} |_{\mathcal{A}_{\mathcal{F}^{\ell}_{K,L}}} & 0
 \end{bmatrix}
 \begin{bmatrix}
   \Vec{\sigma}^{PV}_{\mathcal{A}_{\mathcal{F}^{\ell}_{K,L}}} \\
   \Vec{p}_{\mathcal{A}_{\mathcal{F}^{\ell}_{K,L}}}
 \end{bmatrix}
 =
 \begin{bmatrix}
   0 \\
   f_{\mathcal{F}^{\ell}_{K,L}}
 \end{bmatrix},
 \label{eq:local_problem_to_define_sigma_pv}
\end{equation}
where $f_{\mathcal{F}^{\ell}_{K,L}}$ is defined as
\begin{equation}
f_{\mathcal{F}^{\ell}_{K,L}}
=
\begin{bmatrix}
- \frac{1}{|| \Vec{1}^{\ell} |_{\mathcal{A}^{\ell}_K} ||^2}\Vec{1}^{\ell} |_{\mathcal{A}^{\ell}_K}  \\
  \frac{1}{|| \Vec{1}^{\ell} |_{\mathcal{A}^{\ell}_L} ||^2} \Vec{1}^{\ell} |_{\mathcal{A}^{\ell}_L}
\end{bmatrix}.
\end{equation}
After solving \eqref{eq:local_problem_to_define_sigma_pv}, we set $\Vec{\sigma}^{PV}_{\mathcal{F}^{\ell}_{K,L}} := \Vec{\sigma}^{PV}_{\mathcal{A}_{\mathcal{F}^{\ell}_{K,L}}}|_{\Sigma^{\ell}(\mathcal{F}^{\ell}_{K,L})}$, so that $\Vec{\sigma}^{PV}_{\mathcal{F}^{\ell}_{K,L}}$ is supported on $\mathcal{F}^{\ell}_{K,L}$.
With definition \eqref{eq:local_problem_to_define_sigma_pv}, $\Vec{\sigma}^{PV}_{\mathcal{F}^{\ell}_{K,L}}$ satisfies:
\begin{equation}
(\Vec{q}^{PV}_{\mathcal{A}^{\ell}_K})^T \Mat{D}^{\ell} |_{\mathcal{A}^{\ell}_K, \, \mathcal{F}^{\ell}_{K,L}} \Vec{\sigma}^{PV}_{\mathcal{F}^{\ell}_{K,L}} \neq 0.  
\end{equation}
where $\Vec{q}^{PV}_{\mathcal{A}^{\ell}_{K}}$ is defined in \eqref{eq:definition_q_pv}.
This property is needed to prove the commutativity relation $\Mat{D}^{\ell} \Mat{\Pi}^{\ell}_{\Vec{\sigma}} = \Mat{\Pi}^{\ell}_{p} \Mat{D}^{\ell}$ and to demonstrate the \textit{inf-sup} stability of the coarse space \cite{ml-spectral-coarsening}.
We mention here that for the numerical examples performed with only one coarse degree of freedom per interface, the traces obtained by \eqref{eq:definition-traces} are discarded and $\Vec{\sigma}^{PV}_{\mathcal{F}^{\ell}_{K,L}}$ is the only flux degree of freedom at coarse interface $\mathcal{F}^{\ell}_{K,L}$ that is kept in the coarse space.
If we use more that one flux degree of freedom per interface, we remove any possible linear dependence by applying an SVD to the collected traces at each coarse interface.

We now extend the remaining traces associated with coarse interface $\mathcal{F}^{\ell}_{K,L}$ to the neighboring aggregates, $\mathcal{A}^{\ell}_K$ and $\mathcal{A}^{\ell}_L$.
%
%
To extend $\Vec{\sigma}_{\mathcal{F}^{\ell}_{K,L}}$ to $\mathcal{A}^{\ell}_K$, we solve a local Neumann boundary value problem restricted to the interior of the aggregate.
Specifically, we find $\Vec{\sigma}_{\mathcal{A}^{\ell}_K}$ on $\mathcal{A}^{\ell}_K$ and $\Vec{p}_{\mathcal{A}^{\ell}_K}$ on $\mathcal{A}^{\ell}_K$ such that:
\begin{equation}
 \begin{bmatrix}
     \Mat{M}^{\ell} |_{\mathcal{A}^{\ell}_K} & (\Mat{D}^{\ell} |_{\mathcal{A}^{\ell}_K})^T \\
     \Mat{D}^{\ell} |_{\mathcal{A}^{\ell}_K} & 0
 \end{bmatrix}
 \begin{bmatrix}
   \Vec{\sigma}_{\mathcal{A}^{\ell}_K} \\
   \Vec{p}_{\mathcal{A}^{\ell}_K}
 \end{bmatrix}
 =
 \begin{bmatrix}
   - \Mat{M}^{\ell} |_{\mathcal{A}^{\ell}_K, \, \mathcal{F}^{\ell}_{K,L}} \Vec{\sigma}_{\mathcal{F}^{\ell}_{K,L}} \\
   c_{\mathcal{A}^{\ell}_K} \Vec{q}^{PV}_{\mathcal{A}^{\ell}_K} - \Mat{D}^{\ell} |_{\mathcal{A}^{\ell}_K, \, \mathcal{F}^{\ell}_{K,L}} \Vec{\sigma}_{\mathcal{F}^{\ell}_{K,L}}
 \end{bmatrix}.
 \label{eq:local_neumann_boundary_value_problem}
\end{equation}  
where the non-zero constant $c_{\mathcal{A}^{\ell}_K}$ on the right-hand side of \eqref{eq:local_neumann_boundary_value_problem} is:
\begin{equation}
  c_{\mathcal{A}^{\ell}_K} :=
  (\Vec{q}^{PV}_{\mathcal{A}^{\ell}_K})^T \Mat{D}^{\ell} |_{\mathcal{A}^{\ell}_{K},\, \mathcal{F}^{\ell}_{K,L}} \Vec{\sigma}_{\mathcal{F}^{\ell}_{K,L}}.
\end{equation}
The same procedure is used to extend $\Vec{\sigma}_{\mathcal{F}^{\ell}_{K,L}}$ to the remaining part of aggregate $\mathcal{A}^{\ell}_L$, so that the extended vector is supported on the union of the edges connecting the vertices in $\mathcal{A}^{\ell}_K$, the edges connecting the vertices in $\mathcal{A}^{\ell}_L$, and the edges in $\mathcal{F}^{\ell}_{K,L}$.

To guarantee that the compatibility constraints outlined in \cite{ml-spectral-coarsening} are satisfied, we add linearly independent bubbles to the edge-based (flux) space.
Each bubble is associated with an aggregate $\mathcal{A}^{\ell}$. 
%
For each column $\Vec{q}$ of the matrix $\Mat{P}^{NPV}_{\mathcal{A}^{\ell}}$ appearing in \eqref{eq:structure_pressure_interpolation}, we obtain the bubble basis vector supported on the edges of $\mathcal{A}^{\ell}$ by solving the saddle-point system \eqref{eq:local_neumann_boundary_value_problem} in which we replace the right-hand side with $\begin{bmatrix} 0 \\ \Vec{q} \end{bmatrix}$. 
The resulting vector is then extended by zero outside $\mathcal{A}^{\ell}$.

To summarize, the coarse edge-based space is spanned by the trace extensions---to which we have added the vector  defined by \eqref{eq:local_problem_to_define_sigma_pv}---associated with the coarse interfaces, and by the bubbles associated with the aggregates.
Since the traces are extended into the interior edges of the neighboring aggregates, the resulting interpolation matrix, $(\Mat{P}_{\sigma})^{\ell}_{\ell+1}$, has the lower-triangular form
\begin{equation}
(\Mat{P}_{\sigma})^{\ell}_{\ell+1} := \begin{bmatrix}
\Mat{P}^{\textit{TR}}_\sigma & \\ \Mat{P}^{E}_\sigma & \Mat{P}^B_\sigma
\end{bmatrix},  
\end{equation}
where $\Mat{P}^{\textit{TR}}_{\sigma}$ and $\Mat{P}^{B}_\sigma$ are block-diagonal matrices.
Specifically, in $\Mat{P}^{\textit{TR}}_{\sigma}$, the columns of the $i$th diagonal block are formed by the part of each trace extension associated with the $i$th interface that is supported on the edges forming the interface.
In $\Mat{P}^{B}_\sigma$, the columns of the $i$th diagonal block are formed by the bubbles constructed in the $i$th aggregate.
$\Mat{P}^{E}_{\sigma}$ contains the ``extended'' part of the trace extensions that is supported on the interior edges of the aggregates.
This concludes the construction of the coarse edge-based space.
We refer to \cite{ml-spectral-coarsening} for a thorough description of the construction of the matrix $(\Mat{Q}_{\sigma})^{\ell+1}_{\ell}$ used to obtained the projection operator in the form $(\Mat{\Pi}_{\sigma})^{\ell} := (\Mat{P}_{\sigma})^{\ell}_{\ell+1} (\Mat{Q}_{\sigma})^{\ell+1}_{\ell}$.

%

\section{The piecewise-constant projector $\widetilde{\Mat{\Pi}}^{\ell}$}\label{sec:piecewise-constant-projection}

Recall that when evaluating the nonlinear operator on each level $\ell$, the pressure approximation is first projected to
the space of piecewise-constant functions, with the constants representing the average pressure values in different cells
of that level, cf. \eqref{eq:structure}. Here, the exact definition of the projection will be given. On level $\ell = 0$,
$\widetilde{\Mat{\Pi}}^{0}$ is simply the identity map. Next, on level $\ell > 0$, we first consider the Boolean rectangular sparse matrix
which maps the cell aggregates to the cells, denoted by $(\widetilde{P})_{\ell}^{\ell-1}$.
Specifically, $[(\widetilde{P})_{\ell}^{\ell-1}]_{ij} = 1$ if cell $i$ (at the finer level, $\ell-1$) belongs to aggregate $j$ (at the coarser level, $\ell$);  $[(\widetilde{P})_{\ell}^{\ell-1}]_{ij} = 0$ otherwise.
Furthermore, we multiply them together to get the connection between level $\ell$ and level 0:
\begin{equation}
(\widetilde{P})_{\ell}^{0} := (\widetilde{P})_{\ell}^{\ell-1} (\widetilde{P})_{\ell-1}^{\ell-2} \cdots (\widetilde{P})_{1}^{0}.
\end{equation}
Similarly, we combine the interpolation matrices of the pressure space:
\begin{equation}
(\Mat{P}_p)_{\ell}^{0} := (\Mat{P}_p)_{\ell}^{\ell-1} (\Mat{P}_p)_{\ell-1}^{\ell-2} \cdots (\Mat{P}_p)_{1}^{0}.
\end{equation}
Let $M_p$ be the mass matrix, which is a diagonal matrix with cell volumes on the diagonal: $(M_p)_{K, K} = |\tau_K|$. The piecewise-constant projector is defined as
\begin{equation}
\widetilde{\Mat{\Pi}}^{\ell} := \left( \big( (\widetilde{P})_{\ell}^0 \big)^T M_p (\widetilde{P})_{\ell}^{0} \right)^{-1} \big( (\widetilde{P})_{\ell}^0 \big)^T M_p (\Mat{P}_p)_{\ell}^0.
\label{eq:pwc_projector}
\end{equation}
Notice that $\big( (\widetilde{P})_{\ell}^0 \big)^T M_p (\widetilde{P})_{\ell}^{0}$ is also diagonal, and the $i$th entry on its diagonal is the volume of $\mathcal{A}^{\ell}_i$: $|\mathcal{A}^{\ell}_i| = \sum_{\tau_K\subset \mathcal{A}^{\ell}_i} |\tau_K|$.
In fact, for any vector $\Vec{p}^0 = \left( p^0_K \right)_{\tau_K\in\mathcal{T}}$ on level 0, $\left( \big( (\widetilde{P})_{\ell}^0 \big)^T M_p (\widetilde{P})_{\ell}^{0} \right)^{-1} \big( (\widetilde{P})_{\ell}^0 \big)^T M_p \Vec{p}^0$ is a vector of size number of aggregates, and the $i$th entry equals the weighted average:  $|\mathcal{A}^{\ell}_i|^{-1} \left( \sum_{\tau_K\subset \mathcal{A}^{\ell}_i} |\tau_K| p^0_K\right)$.
Therefore, conceptually what $\widetilde{\Mat{\Pi}}^{\ell}$ does is to first interpolate a given approximation $\Vec{p}^\ell$ on level $\ell$ all the way to the finest level, and then compute the averages of $ (\Mat{P}_p)_{\ell}^0 \Vec{p}^\ell$ over the aggregates on level $\ell$.

An important property of $\widetilde{\Mat{\Pi}}^{\ell} $ is that, for any pressure approximation $\Vec{p}^\ell$ on level $\ell$, if we interpolate $\widetilde{\Mat{\Pi}}^{\ell}(\Vec{p}^\ell)$ all the way to the fine level, evaluate the nonlinear function $\kappa$ in each fine-grid cell, and then compute the average value of $\kappa$ of each aggregate, it is equivalent to evaluating $\kappa$ directly using entries of $\widetilde{\Mat{\Pi}}^{\ell}\Vec{p}^\ell$. That is,
\begin{equation}
\kappa(\widetilde{\Mat{\Pi}}^{\ell}\Vec{p}^\ell) = \left( \big( (\widetilde{P})_{\ell}^0 \big)^T M_p (\widetilde{P})_{\ell}^{0} \right)^{-1}\big( (\widetilde{P})_{\ell}^{0} \big)^T M_p \;\kappa \big((\widetilde{P})_{\ell}^{0} \widetilde{\Mat{\Pi}}^{\ell}\Vec{p}^\ell\big).
\end{equation}
Hence, the nonlinear operator $\blkMat{A}^{\ell}$ in \eqref{eq:structure} can be evaluated directly on level $\ell$ without visiting the finest level during the multigrid iteration.

\section{Proof of Proposition~\ref{prop:assemble}}\label{sec:assemble}

\begin{proof}
In Section~\ref{sec:all-level-M}, we have shown that property \eqref{eq:local-product-decomposition} holds at the finest level, $\ell = 0$.
In this appendix, we assume, without loss of generality, that the proposition holds for some $\ell \geq 0$.
By the principle of induction, it suffices to show that the proposition also holds for some $\ell+1$.

Let us consider an aggregate, $\mathcal{A}^{\ell}$, obtained by agglomerating finer cells at level $\ell$, such that $\mathcal{A}^{\ell} = \bigcup_{i \, \, s.t. \, \, \mathcal{A}^{\ell-1}_i \subset \mathcal{A}^{\ell}}  \mathcal{A}^{\ell-1}_i$.
Let $\Mat{M}^{\ell}_{\mathcal{A}^{\ell}}(\Vec{p})$ be the assembled matrix for the fine-level flux degrees of freedom on the cell aggregate, or equivalently, on the coarse cell at level $\ell+1$.
Next, let $(\Mat{P}_{\sigma, \,\mathcal{A}^{\ell}})_{\ell+1}^\ell$ be the restriction of $(\Mat{P}_\sigma)_{\ell+1}^\ell$ on $\mathcal{A}^{\ell}$.
We coarsen $\Mat{M}_{\mathcal{A}^{\ell}}^\ell(\Vec{p})$ locally through the Galerkin projection restricted to $\mathcal{A}^{\ell}$:
\begin{equation}
\Mat{M}_{\mathcal{A}^{\ell}}^{\ell+1}(\Vec{p}) := \big( (\Mat{P}_{\sigma, \, \mathcal{A}^{\ell}})_{\ell+1}^\ell \big)^T \Mat{M}_{\mathcal{A}^{\ell}}^\ell(\Vec{p})(\Mat{P}_{\sigma, \, \mathcal{A}^{\ell}})_{\ell+1}^\ell.
\end{equation}
Similarly, we assemble $M^{\ell}_{\mathcal{A}^{\ell-1}_i}$ for all $\mathcal{A}_i^{\ell-1} \subset \mathcal{A}^{\ell}$ to form $\Mat{M}_{\mathcal{A}^{\ell}}^\ell$, which is then coarsened to get
\begin{equation}
\Mat{M}_{\mathcal{A}^{\ell}}^{\ell+1} := \big( (\Mat{P}_{\sigma, \, \mathcal{A}^{\ell}})_{\ell+1}^\ell \big)^T \Mat{M}_{\mathcal{A}^{\ell}}^\ell(\Mat{P}_{\sigma, \, \mathcal{A}^{\ell}})_{\ell+1}^\ell.  
\end{equation}
Note that by definition $\widetilde{ \Mat{\Pi} }^{\ell+1} \Vec{p}^{\ell+1}$ is constant in all $\mathcal{A}^{\ell-1}_i \subset \mathcal{A}^{\ell}$, so
\begin{equation}
p_{\mathcal{A}^{\ell}}
:= ( \widetilde{ \Mat{\Pi} }^{\ell+1} \Vec{p}^{\ell+1})|_{\mathcal{A}^{\ell}}
=
(\widetilde{ \Mat{\Pi} }^{\ell+1} \Vec{p}^{\ell+1})|_{\mathcal{A}^{\ell-1}_i}, \quad \forall\,\mathcal{A}^{\ell-1}_i\subset\mathcal{A}^{\ell}.
\end{equation}
Therefore, by \eqref{eq:local-product-decomposition} in the induction hypothesis, we have
\begin{equation}
\Mat{M}_{\mathcal{A}^{\ell-1}_i}^\ell ( \widetilde{ \Mat{\Pi} }^{\ell+1} \Vec{p}^{\ell+1} ) = \frac{1}{\kappa(p_{\mathcal{A}^{\ell}})} M_{\mathcal{A}^{\ell-1}_i}^\ell, \quad \forall\,\mathcal{A}^{\ell-1}_i\subset\mathcal{A}^{\ell}.  
\end{equation}
This means that all the building blocks of $M_{\mathcal{A}^{\ell}}^{\ell} ( \widetilde{ \Mat{\Pi} }^{\ell+1} \Vec{p}^{\ell+1})$ share the same multiplicative factor $1/\kappa(p_{\mathcal{A}^{\ell}})$. Consequently,
\begin{equation}
\begin{split}
\Mat{M}_{\mathcal{A}^{\ell}}^{\ell+1}  
( \widetilde{ \Mat{\Pi} }^{\ell+1} \Vec{p}^{\ell+1} )
& =
\big( (\Mat{P}_{\sigma, \, \mathcal{A}^{\ell}})_{\ell+1}^\ell\big)^T \Mat{M}_{\mathcal{A}^{\ell}}^\ell ( \widetilde{ \Mat{\Pi} }^{\ell+1} \Vec{p}^{\ell+1} ) (\Mat{P}_{\sigma, \, \mathcal{A}^{\ell}})_{\ell+1}^\ell  \\
& = \big( (\Mat{P}_{\sigma, \, \mathcal{A}^{\ell}})_{\ell+1}^\ell\big)^T \bigg( \frac{1}{\kappa(p_{\mathcal{A}^{\ell}})} \Mat{M}_{\mathcal{A}^{\ell}}^\ell \bigg)(\Mat{P}_{\sigma, \, \mathcal{A}^{\ell}})_{\ell+1}^\ell = \frac{1}{\kappa(p_{\mathcal{A}^{\ell}})} M_{\mathcal{A}^{\ell}}^{\ell+1}.
\end{split}
\end{equation}
$\Mat{M}_{\mathcal{A}^{\ell}}^{\ell+1}( \widetilde{ \Mat{\Pi} }^{\ell+1} \Vec{p}^{\ell+1} )$ is a coarse local transmissibility matrix associated with $\mathcal{A}^{\ell}$.
If we assemble $\Mat{M}_{\mathcal{A}^{\ell}}^{\ell+1} ( \widetilde{ \Mat{\Pi} }^{\ell+1} \Vec{p}^{\ell+1} )$ using the global-to-local map for coarse flux degrees of freedom, we get exactly the global Galerkin projection $\Mat{M}^{\ell+1} ( \widetilde{ \Mat{\Pi} }^{\ell+1} \Vec{p}^{\ell+1} ) = \big((\Mat{P}_\sigma)_{\ell+1}^\ell\big)^T \Mat{M}^\ell ( \widetilde{ \Mat{\Pi} }^{\ell+1} \Vec{p}^{\ell+1} ) (\Mat{P}_\sigma)_{\ell+1}^\ell$.
\end{proof}

\section{Proof of Proposition~\ref{prop:Mpsigma}}\label{sec:Mpsigma}

In this section, we drop the superscript $\ell$ denoting the level for simplicity.
We consider two technical lemmas before concluding the proof of Proposition~\ref{prop:Mpsigma}.

The first lemma is a direct consequence of \eqref{eq:local-product-decomposition}.
We remind the reader that for a cell aggregate $\mathcal{A}_L$, we define the aggregate pressure as $p_{\mathcal{A}_L} := (\widetilde{ \Mat{\Pi}}\Vec{p})|_{\mathcal{A}_L}$.
Let us denote by $I_{\mathcal{A}_L}$ the identity matrix whose size is equal to the number of flux degrees of freedom on $\mathcal{A}_L$.
Let $\widehat{\Mat{M}}( \widetilde{\Mat{\Pi}} \Vec{p} )$ and $\widehat{\Mat{K}}(\widetilde{\Pi}\Vec{p})$ be the block-diagonal matrices whose $L$th diagonal block is respectively $\Mat{M}_{\mathcal{A}_L}(\widetilde{ \Mat{\Pi} }\Vec{p})$ and $\kappa(p_{\mathcal{A}_L}) I_{\mathcal{A}_L}$, where $\Mat{M}_{\mathcal{A}_L}(\widetilde{ \Mat{\Pi} }\Vec{p})$ is the local transmissibility matrix introduced in Proposition~\ref{prop:assemble}.
%
%
\begin{lemma}
\label{lemma:Mp}
We have
\begin{equation}
  \widehat{\Mat{M}}(\widetilde{\Pi}\Vec{p}) = \widehat{\Mat{M}} \widehat{\Mat{K}}(\widetilde{\Pi}\Vec{p})^{-1} ,
  \label{eq:lemma_Mp}
\end{equation}
where $\widehat{M}$ is the block-diagonal matrix whose $L$th block is the static matrix $\Mat{M}_{\mathcal{A}_L}$ defined on cell aggregate $\mathcal{A}_L$ and introduced in Proposition~\ref{prop:assemble}.
\end{lemma}
\begin{proof}
Consider the $L$th diagonal block in $\widehat{\Mat{M}}(\widetilde{\Pi}\Vec{p})$.
By definition, this block is equal to $\Mat{M}_{\mathcal{A}_L}(\widetilde{ \Mat{\Pi} }\Vec{p})$.
We know from Proposition~\ref{prop:assemble} that $\Mat{M}_{\mathcal{A}_L}(\widetilde{ \Mat{\Pi} }\Vec{p}) = 1 / \kappa(p_{\mathcal{A}_L}) \Mat{M}_{\mathcal{A}_L}$.
Using this remark with the definitions of $\widehat{\Mat{M}}(\widetilde{\Mat{\Pi} } \Vec{p})$, $\widehat{\Mat{M}}$, and $\widehat{\Mat{K}}(\widetilde{\Pi}\Vec{p})$, we obtain \eqref{eq:lemma_Mp}.
\end{proof}

Using now the notation $\diag( \Vec{u} )$ to denote the diagonal matrix that contains the entries of a vector $\Vec{u}$ on its diagonal, we consider the following lemma:
\begin{lemma}
\label{lemma:Kp}
We have
\begin{equation}
\widehat{\Mat{K}}(\widetilde{\Mat{\Pi}}\Vec{p})^{-1} = \diag\big( \Mat{W}_{\hat{\sigma}p} \Vec{\kappa}_{inv}(\widetilde{\Mat{\Pi}}\Vec{p}) \big),
\end{equation}
where $\Mat{W}_{\hat{\sigma}p}$ is the map from cells to local degrees of freedom, and $\Vec{\kappa}_{inv}(\widetilde{\Mat{\Pi}}\Vec{p})$ is the vector defined in \eqref{eq:definition_inverse_kappa_vector}.
\end{lemma}
\begin{proof}
By definition, $\Mat{W}_{\hat{\sigma}p}$ is a Boolean rectangular matrix such that $[\Mat{W}_{\hat{\sigma}p}]_{ij} = 1$ if the $i$th local flux degree of freedom belongs to cell $j$, and $[\Mat{W}_{\hat{\sigma}p}]_{ij} = 0$ otherwise.
We assume that the local flux degrees of freedom are ordered cell-by-cell.
Therefore, we can write the vector $\Mat{W}_{\hat{\sigma}p} \Vec{\kappa}_{inv}( \widetilde{\Mat{\Pi}} \Vec{p} )$ as the concatenation of subvectors $\Vec{w}_{\mathcal{A}_L} := 1/\kappa( p_{\mathcal{A}_L} ) [1 \, \dots \, 1]^T$ whose size is equal to the number of local flux degrees of freedom of cell aggregate $\mathcal{A}_L$.
This means that $\diag\big( \Mat{W}_{\hat{\sigma}p} \Vec{\kappa}_{inv}(\widetilde{\Mat{\Pi}}\Vec{p}) \big)$ is a block-diagonal matrix whose $L$th diagonal block is a diagonal submatrix equal to $\diag( \Vec{w}_{\mathcal{A}_L} )$.
We can rewrite these diagonal blocks as $\diag( \Vec{w}_{\mathcal{A}_L} ) = 1/\kappa( p_{\mathcal{A}_L} ) \diag( [1 \, \dots \, 1]^T ) = 1/\kappa( p_{\mathcal{A}_L} ) \Mat{I}_{\mathcal{A}_L}$. So, $\diag\big( \Mat{W}_{\hat{\sigma}p} \Vec{\kappa}_{inv}(\widetilde{\Mat{\Pi}}\Vec{p}) \big)$ is a block-diagonal matrix whose $L$th diagonal block is equal to $1/\kappa( p_{\mathcal{A}_L} ) \Mat{I}_{\mathcal{A}_L}$, which corresponds to the inverse of the matrix $\widehat{\Mat{K}}(\widetilde{\Pi}\Vec{p})$ defined before Lemma~\ref{lemma:Mp}.
\end{proof}

We are now ready to prove Proposition~\ref{prop:Mpsigma}.
The global matrix $\Mat{M}(\widetilde{\Mat{\Pi}}\Vec{p})$ is obtained by assembling the local transmissibility matrix using the global-to-local map for the flux degrees of freedom, $\Mat{W}_{\hat{\sigma}\sigma}$, as
\begin{equation}
  \Mat{M}(\widetilde{\Mat{\Pi}}\Vec{p}) = \Mat{W}_{\hat{\sigma}\sigma}^T \widehat{\Mat{M}}(\widetilde{\Mat{\Pi}}\Vec{p}) \Mat{W}_{\hat{\sigma}\sigma}.
  \label{eq:appendix_global_matrix}
\end{equation}
Using \eqref{eq:appendix_global_matrix} and Lemma~\ref{lemma:Mp}, we can write
\begin{equation}
  \begin{split}
    \Mat{M}(\widetilde{\Mat{\Pi}}\Vec{p})\Vec{\sigma} &= \Mat{W}_{\hat{\sigma}\sigma}^T \widehat{\Mat{M}} (\widetilde{\Mat{\Pi} }\Vec{p}) \Mat{W}_{\hat{\sigma}\sigma} \Vec{\sigma} \\
    &= \Mat{W}_{\hat{\sigma}\sigma}^T \widehat{\Mat{M}} \widehat{\Mat{K}}(\widetilde{\Mat{\Pi} }\Vec{p})^{-1} \Mat{W}_{\hat{\sigma}\sigma} \Vec{\sigma}.
  \end{split}
\end{equation}
Next, we use Lemma~\ref{lemma:Kp} to obtain
\begin{equation}
  \Mat{M}(\widetilde{\Mat{\Pi}}\Vec{p})\Vec{\sigma} = \Mat{W}_{\hat{\sigma}\sigma}^T \widehat{\Mat{M}}  \diag \big( \Mat{W}_{\hat{\sigma}p} \Vec{\kappa}_{inv}(\widetilde{\Mat{\Pi}}\Vec{p}) \big) \Mat{W}_{\hat{\sigma}\sigma} \Vec{\sigma}.
  \label{eq:appendix_intermediate_result}
\end{equation}
Let us consider now two vectors $\Vec{u}, \Vec{v}\in\mathbb{R}^{m}$. We remark that
\begin{equation}
 \diag(\Vec{v}) \Vec{u}  = \diag(\Vec{u}) \Vec{v} .
\end{equation}
Applying this remark to the vectors $\Mat{W}_{\hat{\sigma}p} \Vec{\kappa}_{inv}(\widetilde{\Mat{\Pi}}\Vec{p})$ and $\Mat{W}_{\hat{\sigma}\sigma} \Vec{\sigma}$ in \eqref{eq:appendix_intermediate_result} yields
\begin{equation}
  \begin{split}
    \Mat{M}(\widetilde{\Mat{\Pi}}\Vec{p})\Vec{\sigma}
    &= \Mat{W}_{\hat{\sigma}\sigma}^T \widehat{\Mat{M}} \diag \big( \Mat{W}_{\hat{\sigma}\sigma} \Vec{\sigma}\ \big)  \Mat{W}_{\hat{\sigma}p} \Vec{\kappa}_{inv}(\widetilde{\Mat{\Pi}}\Vec{p}) \\
    &= \Mat{N}(\Vec{\sigma}) \Vec{\kappa}_{inv}(\widetilde{\Mat{\Pi}}\Vec{p}),
  \end{split}
\end{equation}
where $\Mat{N}(\Vec{\sigma})$ is defined in Proposition~\ref{prop:Mpsigma}. This concludes the proof.


\bibliography{authorList,journalAbbreviations,proceedingCollectionNames,publisherNames,references_NEW}

\begin{thebibliography}{10}
\expandafter\ifx\csname url\endcsname\relax
  \def\url#1{\texttt{#1}}\fi
\expandafter\ifx\csname urlprefix\endcsname\relax\def\urlprefix{URL }\fi
\expandafter\ifx\csname href\endcsname\relax
  \def\href#1#2{#2} \def\path#1{#1}\fi

\bibitem{keyes06}
D.~E. Keyes, D.~R. Reynolds, C.~S. Woodward, Implicit solvers for large-scale
  nonlinear problems, J. Phys.: Conf. Ser. 46 (2006) 433--442.
\newblock \href {http://dx.doi.org/10.1088/1742-6596/46/1/060}
  {\path{doi:10.1088/1742-6596/46/1/060}}.

\bibitem{luo2015}
P.~Luo, C.~Rodrigo, F.~J. Gaspar, C.~W. Oosterlee, Multigrid method for
  nonlinear poroelasticity equations, Comput. Vis. Sci. 17~(5) (2015) 255--265.
\newblock \href {http://dx.doi.org/10.1007/s00791-016-0260-8}
  {\path{doi:10.1007/s00791-016-0260-8}}.

\bibitem{christensen16}
M.~l.~C. Christensen, K.~L. Eskildsen, A.~P. Engsig-Karup, M.~A. Wakefield,
  Nonlinear multigrid for reservoir simulation, SPE J. 21~(3) (2016) 888--898.
\newblock \href {http://dx.doi.org/10.2118/178428-PA}
  {\path{doi:10.2118/178428-PA}}.

\bibitem{toft18}
R.~Toft, K.-A. Lie, O.~M{\o}yner,
  \href{{https://ojs.bibsys.no/index.php/NIK/article/view/503}}{ Full
  approximation scheme for reservoir simulation}, in: Proceedings - Norsk
  Informatikkonferanse, Oslo, Norway, 2018.

\bibitem{dumett02}
M.~A. Dumett, P.~S. Vassilevski, C.~S. Woodward,
  \href{{https://computing.llnl.gov/casc/nsde/pubs/DumettVassilevskiWoodward02.pdf}}{
  A multigrid method for nonlinear unstructured finite element elliptic
  equations}, Preprint UCRL-JC-150513, Lawrence Livermore National Laboratory
  (2002).

\bibitem{jones03}
J.~E. Jones, P.~S. Vassilevski, C.~S. Woodward, {Nonlinear Schwarz-FAS} methods
  for unstructured finite element elliptic problems, in: K.-J. Bathe (Ed.),
  Computational Fluid and Solid Mechanics 2003, M.I.T. Conferences on
  Computational Fluid and Solid Mechanics, 2003, pp. 2008--2011.
\newblock \href {http://dx.doi.org/10.1016/B978-008044046-0.50492-9}
  {\path{doi:10.1016/B978-008044046-0.50492-9}}.

\bibitem{christensen18}
M.~l.~C. Christensen, P.~S. Vassilevski, U.~Villa, Nonlinear multigrid solvers
  exploiting {AMG}e coarse spaces with approximation properties, J. Comput.
  Appl. Math. 340 (2018) 691--708.
\newblock \href {http://dx.doi.org/10.1016/j.cam.2017.10.029}
  {\path{doi:10.1016/j.cam.2017.10.029}}.

\bibitem{barker17}
A.~T. Barker, C.~S. Lee, P.~S. Vassilevski, Spectral upscaling for graph
  {L}aplacian problems with application to reservoir simulation, SIAM J. Sci.
  Comput. 39~(5) (2017) S323--S346.
\newblock \href {http://dx.doi.org/10.1137/16M1077581}
  {\path{doi:10.1137/16M1077581}}.

\bibitem{ml-spectral-coarsening}
A.~T. Barker, S.~V. Gelever, C.~S. Lee, S.~Osborn, P.~S. Vassilevski,
  Multilevel spectral coarsening for graph {L}aplacian problems with
  application to reservoir simulation (mar 2020).
\newblock \href {http://arxiv.org/abs/2003.04423} {\path{arXiv:2003.04423}}.

\bibitem{coats2000}
K.~H. Coats, A note on {IMPES} and some {IMPES}-based simulation models, SPE J.
  5~(3) (2000) 245--251.
\newblock \href {http://dx.doi.org/10.2118/65092-PA}
  {\path{doi:10.2118/65092-PA}}.

\bibitem{EymGalHer00}
R.~Eymard, T.~Gallou{\"{e}}t, R.~Herbin, Finite volume methods, in: P.~G.
  Ciarlet, J.-L. Lions (Eds.), Solution of Equation in {$R^n$} (Part 3),
  Techniques of Scientific Computing (Part 3), Vol.~7 of Handbook of Numerical
  Analysis, Elsevier, 2000, pp. 713--1018.
\newblock \href {http://dx.doi.org/10.1016/S1570-8659(00)07005-8}
  {\path{doi:10.1016/S1570-8659(00)07005-8}}.

\bibitem{KarDur16}
M.~Karimi-Fard, L.~J. Durlofsky, A general gridding, discretization, and
  coarsening methodology for modeling flow in porous formations with discrete
  geological features, Adv. Water Resour. 96 (2016) 354--372.
\newblock \href {http://dx.doi.org/10.1016/j.advwatres.2016.07.019}
  {\path{doi:10.1016/j.advwatres.2016.07.019}}.

\bibitem{BerPle94}
A.~Berman, R.~J. Plemmons, Nonnegative Matrices in the Mathematical Sciences,
  Society for Industrial and Applied Mathematics, 1994.
\newblock \href {http://dx.doi.org/10.1137/1.9781611971262}
  {\path{doi:10.1137/1.9781611971262}}.

\bibitem{Dro14}
J.~Droniou, Finite volume schemes for diffusion equations: {I}ntroduction to
  and review of modern methods, Math. Models Methods Appl. Sci. 24~(08) (2014)
  1575--1619.
\newblock \href {http://dx.doi.org/10.1142/S0218202514400041}
  {\path{doi:10.1142/S0218202514400041}}.

\bibitem{brandt77}
A.~Brandt, Multi-level adaptive solutions to boundary-value problems, Math.
  Comp. 31~(138) (1977) 333--390.
\newblock \href {http://dx.doi.org/10.2307/2006422}
  {\path{doi:10.2307/2006422}}.

\bibitem{henson2003multigrid}
V.~E. Henson, Multigrid methods for nonlinear problems: {A}n overview, in:
  C.~A. Bouman, R.~L. Stevenson (Eds.), Computational Imaging, Vol. 5016 of
  Proceedings of SPIE, 2003, pp. 36--48.
\newblock \href {http://dx.doi.org/10.1117/12.499473}
  {\path{doi:10.1117/12.499473}}.

\bibitem{karypis1998fast}
G.~Karypis, V.~Kumar, A fast and highly quality multilevel scheme for
  partitioning irregular graphs, SIAM J. Sci. Comput. 20~(1) (1998) 359--392.
\newblock \href {http://dx.doi.org/10.1137/S1064827595287997}
  {\path{doi:10.1137/S1064827595287997}}.

\bibitem{Vas08}
P.~S. Vassilevski, Multilevel block factorization preconditioners:
  {M}atrix-based analysis and algorithms for solving finite element equations,
  Springer-Verlag, 2008.
\newblock \href {http://dx.doi.org/10.1007/978-0-387-71564-3}
  {\path{doi:10.1007/978-0-387-71564-3}}.

\bibitem{block-diagonal}
E.~de~Sturler, J.~Liesen, Block-diagonal and constraint preconditioners for
  nonsymmetric indefinite linear systems. {Part I}: {T}heory, SIAM J. Sci.
  Comput. 26~(5) (2005) 1598--1619.
\newblock \href {http://dx.doi.org/10.1137/S1064827502411006}
  {\path{doi:10.1137/S1064827502411006}}.

\bibitem{benzi05}
M.~Benzi, G.~H. Golub, J.~Liesen, Numerical solution of saddle point problems,
  Acta Numer. 14 (2005) 1--137.
\newblock \href {http://dx.doi.org/10.1017/S0962492904000212}
  {\path{doi:10.1017/S0962492904000212}}.

\bibitem{henson02}
V.~E. Henson, U.~M. Yang, \textit{BoomerAMG}: {A} parallel algebraic multigrid
  solver and preconditioner, Appl. Numer. Math. 41~(1) (2002) 155--177.
\newblock \href {http://dx.doi.org/10.1016/S0168-9274(01)00115-5}
  {\path{doi:10.1016/S0168-9274(01)00115-5}}.

\bibitem{hypre}
\textit{hypre}: Scalable linear solvers and multigrid methods,
  \url{http://www.llnl.gov/casc/hypre}.

\bibitem{dobrev18}
V.~A. Dobrev, {\relax Tz}.~V. Kolev, C.~S. Lee, V.~Z. Tomov, P.~S. Vassilevski,
  Algebraic hybridization and static condensation with application to scalable
  {$H$(div)} preconditioning, SIAM J. Sci. Comput. 41~(3) (2019) B425--B447.
\newblock \href {http://dx.doi.org/10.1137/17M1132562}
  {\path{doi:10.1137/17M1132562}}.

\bibitem{arnold85}
D.~N. Arnold, F.~Brezzi, Mixed and nonconforming finite element methods :
  implementation, postprocessing and error estimates, ESAIM Math. Model. Numer.
  Anal. 19~(1) (1985) 7--32.
\newblock \href {http://dx.doi.org/10.1051/m2an/1985190100071}
  {\path{doi:10.1051/m2an/1985190100071}}.

\bibitem{lee17}
C.~S. Lee, P.~S. Vassilevski, Parallel solver for {${H}$}(div) problems using
  hybridization and {AMG}, in: C.-O. Lee, X.-C. Cai, D.~E. Keyes, H.~H. Kim,
  A.~Klawonn, E.-J. Park, O.~B. Widlund (Eds.), Domain Decomposition Methods in
  Science and Engineering XXIIII, Vol. 116 of Lecture Notes in Computational
  Science and Engineering, 2017, pp. 69--80.
\newblock \href {http://dx.doi.org/10.1007/978-3-319-52389-7_6}
  {\path{doi:10.1007/978-3-319-52389-7_6}}.

\bibitem{mfem}
{MFEM}: {M}odular finite element methods library, \url{http://mfem.org}.

\bibitem{smoothg}
{smoothG}: {M}ixed graph {L}aplacian upscaling and solvers,
  \url{https://github.com/LLNL/smoothG}.

\bibitem{glvis}
{GLVis}: {OpenGL} finite element visualization tool, \url{http://glvis.org}.

\bibitem{EggModel}
J.~D. Jansen, R.-M. Fonseca, S.~Kahrobaei, M.~M. Siraj, G.~M. Van~Essen,
  P.~M.~J. Van~den Hof, The egg model - a geological ensemble for reservoir
  simulation, Geosci. Data J. 1~(2) (2014) 192--195.
\newblock \href {http://dx.doi.org/10.1002/gdj3.21}
  {\path{doi:10.1002/gdj3.21}}.

\bibitem{EggModelData}
J.~D. Jansen, R.-M. Fonseca, S.~Kahrobaei, M.~M. Siraj, G.~M. Van~Essen,
  P.~M.~J. Van~den Hof, The {E}gg {M}odel - data files, TU Delft. Dataset
  (2013).
\newblock \href
  {http://dx.doi.org/10.4121/uuid:916c86cd-3558-4672-829a-105c62985ab2}
  {\path{doi:10.4121/uuid:916c86cd-3558-4672-829a-105c62985ab2}}.

\bibitem{spe10}
M.~A. Christie, M.~J. Blunt, Tenth {SPE} comparative solution project: {A}
  comparison of upscaling techniques, SPE Reserv. Eval. Eng. 4~(4) (2001)
  308--317.
\newblock \href {http://dx.doi.org/10.2118/72469-PA}
  {\path{doi:10.2118/72469-PA}}.

\bibitem{mckeon87}
T.~J. McKeon, W.-S. Chu, A multigrid model for steady flow in partially
  saturated porous media, Water Resour. Res. 23~(4) (1987) 542--550.
\newblock \href {http://dx.doi.org/10.1029/WR023i004p00542}
  {\path{doi:10.1029/WR023i004p00542}}.

\bibitem{pasciak2008exact}
J.~E. Pasciak, P.~S. Vassilevski, Exact de {R}ham sequences of spaces defined
  on macro-elements in two and three spatial dimensions, SIAM J. Sci. Comput.
  30~(5) (2008) 2427--2446.
\newblock \href {http://dx.doi.org/10.1137/070698178}
  {\path{doi:10.1137/070698178}}.

\end{thebibliography}

\end{document}